\documentclass{scrartcl}

\usepackage[utf8]{luainputenc}
\usepackage[USenglish]{babel}
\usepackage{csquotes}

\usepackage[a4paper,top=27mm,bottom=20mm,inner=25mm,outer=20mm]{geometry}
\usepackage[plainpages=false,pdfpagelabels,hidelinks,unicode]{hyperref}

\usepackage[%
  backend=bibtex,bibencoding=ascii,
  style=numeric-comp,
  giveninits=true, uniquename=init, 
  natbib=true,
  url=true,
  doi=true,
  isbn=false,
  backref=false,
  maxnames=99,
  ]{biblatex}
\usepackage{amsmath}
\allowdisplaybreaks
\usepackage{amssymb}
\usepackage{commath}
\usepackage{mathtools}
\usepackage{bbm}
\usepackage{nicefrac}

\usepackage{siunitx}

\usepackage{amsthm}
\usepackage{thmtools}
\usepackage{etoolbox}
\makeatletter
\patchcmd{\thmt@setheadstyle}
 {\bgroup\thmt@space}
 {\thmt@space}
 {}{}
\patchcmd{\thmt@setheadstyle}
 {\egroup\fi}
 {\fi}
 {}{}
\makeatother
\declaretheoremstyle[
  bodyfont=\normalfont\itshape,
  headformat=\NAME\ \NUMBER\NOTE,
]{myplain}
\declaretheoremstyle[
  headformat=\NAME\ \NUMBER\NOTE,
]{mydefinition}
\newcommand{\envqed}{{\lower-0.3ex\hbox{$\triangleleft$}}}
\declaretheorem[style=myplain,numberwithin=section]{theorem}
\declaretheorem[style=myplain,numberlike=theorem]{lemma}

\declaretheorem[style=myplain,numberlike=theorem]{corollary}
\declaretheorem[style=mydefinition,numberlike=theorem,qed=\envqed]{definition}
\declaretheorem[style=mydefinition,numberlike=theorem,qed=\envqed]{remark}
\declaretheorem[style=mydefinition,numberlike=theorem,qed=\envqed]{example}

\usepackage{color}
\usepackage{graphicx}
\usepackage[small]{caption}
\usepackage{subcaption}

\begingroup\expandafter\expandafter\expandafter\endgroup
\expandafter\ifx\csname pdfsuppresswarningpagegroup\endcsname\relax
\else
\fi

\usepackage{booktabs}
\usepackage{rotating}
\usepackage{multirow}

\usepackage{enumitem}

\usepackage{ifluatex}
\ifluatex
  \usepackage[no-math]{fontspec}
\else
  \usepackage[T1]{fontenc}
\fi
\usepackage{newpxtext,newpxmath}

\newcommand{\R}{\mathbb{R}}
\renewcommand{\C}{\mathbb{C}}
\newcommand{\scp}[2]{\left\langle{#1,\, #2}\right\rangle}
\newcommand{\I}{\operatorname{I}}
\newcommand{\id}{\operatorname{id}}
\newcommand{\Dinv}{J}
\newcommand{\mDinv}{\tilde{J}}
\newcommand{\F}{F}

\renewcommand{\Re}{\operatorname{Re}}

\DeclarePairedDelimiterX\newset[1]\lbrace\rbrace{\setaux #1||\endsetaux}
\def\setaux#1|#2|#3\endsetaux{\if\relax\detokenize{#2}\relax #1 \else #1 \;\delimsize\vert\; #2 \fi}
\renewcommand{\set}[1]{\newset*{#1}}

\let\epsilon\varepsilon
\let\phi\varphi
\let\rho\varrho

\newcommand{\kernel}{\operatorname{ker}}
\newcommand{\image}{\operatorname{im}}
\newcommand{\spann}{\operatorname{span}}

\newcommand{\dt}{\Delta t}
\renewcommand{\vec}[1]{\pmb{#1}}
\newcommand{\osc}{\vec{o}}

\newenvironment{keywords}{\par\textbf{Key words.}}{\par}
\newenvironment{AMS}{\par\textbf{AMS subject classification.}}{\par}

\title{A New Class of \texorpdfstring{$A$}{A} Stable Summation by Parts Time Integration Schemes with Strong Initial Conditions}
\author{Hendrik Ranocha, Jan Nordström}
\date{December 23, 2020}

\makeatletter
\hypersetup{pdfauthor={\@author}}
\hypersetup{pdftitle={\@title}}
\makeatother

\begin{document}

\maketitle

\begin{abstract}
  Since integration by parts is an important tool when deriving energy or entropy estimates for differential equations, one may conjecture that some form of summation by parts (SBP) property is involved in provably stable numerical methods. This article contributes to this topic by proposing a novel class of $A$ stable SBP time integration methods which can also be reformulated as implicit Runge-Kutta methods. In contrast to existing SBP time integration methods using simultaneous approximation terms to impose the initial condition weakly, the new schemes use a projection method to impose the initial condition strongly without destroying the SBP property. The new class of methods includes the classical Lobatto~IIIA collocation method, not previously formulated as an SBP scheme. Additionally, a related SBP scheme including the classical Lobatto~IIIB collocation method is developed.

\end{abstract}



\section{Introduction}

Based on the fact that integration by parts plays a major role in the
development of energy and entropy estimates for initial boundary value
problems, one may conjecture that the
summation by parts (SBP) property \cite{svard2014review,fernandez2014review}
is a key factor in provably stable schemes.
Although it is complicated to formulate
such a conjecture mathematically, there are
several attempts to unify stable methods in the framework of
summation by parts schemes, starting from the origin of SBP operators in
finite difference methods \cite{kreiss1974finite,strand1994summation}
and ranging from finite volume \cite{nordstrom2001finite,nordstrom2003finite}
and discontinuous Galerkin methods \cite{gassner2013skew} to flux
reconstruction schemes \cite{ranocha2016summation}.

Turning to SBP methods in time \cite{nordstrom2013summation,lundquist2014sbp,boom2015high},
a class of linearly and nonlinearly stable SBP schemes has
been constructed and studied in this context, see also
\cite{ruggiu2018pseudo,ruggiu2020eigenvalue,linders2020properties}.
If the underlying quadrature is chosen as Radau or Lobatto quadrature,
these Runge-Kutta schemes are exactly the classical Radau~IA, Radau~IIA, and
Lobatto~IIIC methods \cite{ranocha2019some}.
Having the conjecture ``stability results require an SBP structure''
in mind, this article provides additional insights to this topic by
constructing new classes of SBP schemes, which reduce to the classical
Lobatto~IIIA and Lobatto~IIIB methods if that quadrature rule is used.
Consequently, all $A$ stable classical Runge-Kutta methods based on
Radau and Lobatto quadrature rules can be formulated in the framework of
SBP operators.
Notably, instead of using simultaneous approximation terms (SATs)
\cite{carpenter1994time,carpenter1999stable}
to impose initial conditions weakly, these new schemes use a strong
imposition of initial conditions in combination with a projection method
\cite{olsson1995summationI,olsson1995summationII,mattsson2018improved}.

By mimicking integration by parts at a discrete level, the stability of
SBP methods can be obtained in a straightforward way by mimicking the
continuous analysis.
All known SBP time integration methods are implicit and their
stability does not depend on the size of the time step.
In contrast, the stability
analysis of explicit time integration methods can use techniques
similar to summation by parts, but the analysis is in general more complicated
and restricted to sufficiently small time steps
\cite{ranocha2018L2stability,sun2017stability,sun2019strong}.
Since there are strict stability limitations for
explicit methods, especially for nonlinear problems
\cite{ranocha2020strong,ranocha2020energy},
an alternative to stable fully implicit methods
is to modify less expensive (explicit or not fully implicit)
time integration schemes to get the desired stability results
\cite{ketcheson2019relaxation,ranocha2020relaxation,ranocha2020general,ranocha2020relaxationHamiltonian,offner2018artificial,sun2019enforcing}.

This article is structured as follows. At first, the existing class of
SBP time integration methods is introduced in Section~\ref{sec:known-results},
including a description of the related stability properties.
Thereafter, the novel SBP time integration methods are
proposed in Section~\ref{sec:new-schemes}. Their stability
properties are studied and the relation to Runge-Kutta methods is
described. In particular, the Lobatto~IIIA and Lobatto~IIIB methods
are shown to be recovered using this framework.
Afterwards, results of numerical experiments demonstrating the
established stability properties are reported in
Section~\ref{sec:numerical-experiments}. Finally, the findings of this
article are summed up and discussed in Section~\ref{sec:summary}.

\section{Known Results for SBP Schemes}
\label{sec:known-results}

Consider an ordinary differential equation (ODE)
\begin{equation}
\label{eq:ode}
  \forall t \in (0,T)\colon \quad
  u'(t) = f(t, u(t)),
  \qquad
  u(0) = u_0,
\end{equation}
with solution $u$ in a Hilbert space.
Summation by parts schemes approximate the solution
on a finite grid $0 \le \tau_1 < \dots < \tau_s \le T$ pointwise as
$\vec{u}_i = u(\tau_i)$ and $\vec{f}_i = f(\tau_i, \vec{u}_i)$.
Although the grid does not need to be ordered for general SBP schemes, we impose
this restriction to simplify the presentation. (An unordered grid can always be
transformed into an ordered one by a permutation of the grid indices.)
The SBP operators can be defined as follows, cf.
\cite{svard2014review,fernandez2014review,fernandez2014generalized}.
\begin{definition}
  A first derivative SBP operator of order $p$ on $[0, T]$ consists of
  \begin{itemize}
    \item
    a discrete operator $D$ approximating the derivative
    $D \vec{u} \approx u'$ with order of accuracy $p$,

    \item
    a symmetric and positive definite discrete quadrature matrix $M$
    approximating the $L^2$ scalar product
    $\vec{u}^T M \vec{v} \approx \int_{0}^{T} u(\tau) v(\tau) \dif \tau$,

    \item
    and interpolation vectors $\vec{t}_L, \vec{t}_R$ approximating the
    boundary values as
    $\vec{t}_L^T \vec{u} \approx u(0)$, $\vec{t}_R^T \vec{u} \approx u(T)$
    with order of accuracy at least $p$, such that the SBP property
    \begin{equation}
    \label{eq:SBP}
      M D + (M D)^T = \vec{t}_R \vec{t}_R^T - \vec{t}_L \vec{t}_L^T
    \end{equation}
    holds.
  \end{itemize}
\end{definition}

\begin{remark}
  There are analogous definitions of SBP operators for second or higher
  order derivatives
  \cite{mattsson2004summation,mattsson2014diagonal,ranocha2020broad}.
  In this article, only first derivative SBP operators are considered.
\end{remark}

\begin{remark}
  The quadrature matrix $M$ is sometimes called norm matrix (since it
  induces a norm via a scalar product) or mass matrix (in a finite element
  context).
\end{remark}

Because of the SBP property \eqref{eq:SBP}, SBP operators
mimic integration by parts discretely via
\begin{equation}
\label{eq:SBP-IBP}
\begin{array}{*3{>{\displaystyle}c}}
  \underbrace{
    \vec{u}^T M (D \vec{v})
    + (D \vec{u})^T M \vec{v}
  }
  & = &
  \underbrace{
    (\vec{t}_R^T \vec{u})^T (\vec{t}_R^T \vec{v}) - (\vec{t}_L^T \vec{u})^T (\vec{t}_L^T \vec{v}),
  }
  \\
  \rotatebox{90}{$\!\approx\;$}
  &&
  \rotatebox{90}{$\!\!\approx\;$}
  \\
  \overbrace{
    \int_{0}^{T} u(\tau) \, v'(\tau) \dif \tau
    + \int_{0}^{T} u'(\tau) \, v(\tau) \dif \tau
  }
  & = &
  \overbrace{
    u(T) v(T) - u(0) v(0)
  }.
\end{array}
\end{equation}
However, this mimetic property does not
suffice for the derivations to follow. Nullspace consistency will
be used as an additional required mimetic property. This novel property was
introduced in \cite{svard2019convergence} and has been a key factor in
\cite{linders2020properties,ranocha2020discrete}.

\begin{definition}
  A first derivative SBP operator $D$ is nullspace consistent,
  if the nullspace (kernel) of $D$ satisfies
  $\kernel D = \spann \set{\vec{1}}$.
\end{definition}
Here, $\vec{1}$ denotes the discrete grid function with value unity at every
node.

\begin{remark}
  Every first derivative operator $D$ (which is at least first order
  accurate) maps constants to zero, i.e.\ $D \vec{1} = \vec{0}$.
  Hence, the kernel of $D$ always satisfies
  $\spann \set{\vec{1}} \leq \kernel D$.
  Here and in the following, $\leq$ denotes the subspace relation of
  vector spaces.
  If $D$ is not nullspace consistent, there are more discrete grid
  functions besides constants which are mapped to zero (which makes
  it inconsistent with $\partial_t$). Then,
  $\ker D \neq \spann \set{\vec{1}}$ and undesired behavior
  can occur, cf.\
  \cite{svard2019convergence,linders2020properties,ranocha2019some,svard2020convergence}.
\end{remark}

An SBP time discretization of \eqref{eq:ode} using SATs
to impose the initial condition weakly is
\cite{nordstrom2013summation,lundquist2014sbp,boom2015high}
\begin{equation}
\label{eq:SBP-SAT}
  D \vec{u} = \vec{f} + M^{-1} \vec{t}_L \bigl( u_0 - \vec{t}_L^T \vec{u} \bigr).
\end{equation}
The numerical solution $u_+$ at $t=T$ is given by $u_+ = \vec{t}_R^T \vec{u}$,
where $\vec{u}$ solves \eqref{eq:SBP-SAT}.

\begin{remark}
  The interval $[0,T]$ can be partitioned into multiple subintervals/blocks
  such that multiple steps of this procedure can be used sequentially
  \cite{lundquist2014sbp}.
\end{remark}

In order to guarantee that \eqref{eq:SBP-SAT} can be solved for a
linear scalar problem, $D + \sigma M^{-1} \vec{t}_L \vec{t}_L^T$
must be invertible, where $\sigma$ is a real parameter usually
chosen as $\sigma = 1$. The following result has been obtained in
\cite[Lemma~2]{linders2020properties}.
\begin{theorem}
\label{thm:nullspace-consistent}
  If $D$ is a first derivative SBP operator,
  $D + M^{-1} \vec{t}_L \vec{t}_L^T$
  is invertible if and only if $D$ is nullspace consistent.
\end{theorem}
\begin{remark}
  In \cite{ruggiu2018pseudo,ruggiu2020eigenvalue}, it was explicitly shown
  how to prove that $D + M^{-1} \vec{t}_L \vec{t}_L^T$ is invertible in the
  pseudospectral/polynomial and finite difference case.
\end{remark}

As many other one-step time integration schemes, SBP-SAT schemes \eqref{eq:SBP-SAT}
can be characterized as Runge-Kutta methods, given by their Butcher
coefficients \cite{hairer2008solving,butcher2016numerical}
\begin{equation}
\label{eq:butcher}
\renewcommand{\arraystretch}{1.2}
\begin{array}{c | c}
  c & A
  \\ \hline
    & b^T
\end{array},
\end{equation}
where $A \in \R^{s \times s}$ and $b, c \in \R^s$. For \eqref{eq:ode}, a step
from $u_0$ to $u_+ \approx u(\dt)$ is given by
\begin{equation}
\label{eq:RK-step}
  u_i
  =
  u_0 + \dt \sum_{j=1}^{s} a_{ij} \, f(c_j \dt, u_j),
  \qquad
  u_+
  =
  u_0 + \dt \sum_{i=1}^{s} b_{i} \, f(c_i \dt, u_i).
\end{equation}
Here, $u_i$ are the stage values of the Runge-Kutta method.
The following characterization of \eqref{eq:SBP-SAT} as Runge-Kutta method
was given in \cite{boom2015high}.
\begin{theorem}
\label{thm:SBP-SAT-as-RK}
  Consider a first derivative SBP operator $D$.
  If $D + M^{-1} \vec{t}_L \vec{t}_L^T$ is invertible, \eqref{eq:SBP-SAT} is equivalent
  to an implicit Runge-Kutta method with the Butcher coefficients
  \begin{equation}
  \label{eq:SBP-SAT-as-RK}
  \begin{aligned}
    A
    &= \frac{1}{T} (D + M^{-1} \vec{t}_L \vec{t}_L^T)^{-1}
    = \frac{1}{T} (M D + \vec{t}_L \vec{t}_L^T)^{-1} M,
    \\
    b &= \frac{1}{T} M \vec{1},
    \phantom{+ \vec{t}_L \vec{t}_L^T)^{-1} M}
    c = \frac{1}{T} (\tau_1, \dots, \tau_s)^T.
  \end{aligned}
  \end{equation}
\end{theorem}
The factor $\frac{1}{T}$ is needed since the Butcher coefficients
of a Runge-Kutta method are normalized to the interval $[0,1]$.

Next, we recall some classical stability properties of Runge-Kutta methods
for linear problems, cf. \cite[Section~IV.3]{hairer2010solving}.
The absolute value of solutions of the scalar linear ODE
\begin{equation}
\label{eq:scalar-test-ode}
  u'(t) = \lambda u(t),
  \quad
  u(0) = u_0 \in \C,
  \quad
  \lambda \in \C,
\end{equation}
cannot increase if $\Re \lambda \leq 0$.
The numerical solution after one time step of a Runge-Kutta method with Butcher
coefficients $A, b, c$ is $u_+ = R(\lambda \, \Delta t) u_0$, where
\begin{equation}
\label{eq:stability-function}
  R(z)
  =
  1 + z b^T (\I - z A)^{-1} \vec{1}
  =
  \frac{\det(\I - z A + z \vec{1} b^T)}{\det(\I - z A)}
\end{equation}
is the \emph{stability function} of the Runge-Kutta method.
The stability property of the ODE is mimicked discretely as
$\abs{u_+} \leq \abs{u_0}$ if $\abs{R(\lambda \, \Delta t)} \leq 1$.
\begin{definition}
  A Runge-Kutta method with stability function $\abs{R(z)} \leq 1$
  for all $z \in \C$ with $\Re(z) \leq 0$ is $A$ stable.
  The method is $L$ stable, if it is $A$ stable and $\lim_{z \to \infty} R(z) = 0$.
\end{definition}
Hence, $A$ stable methods are stable for every time step
$\Delta t > 0$ and $L$ stable methods damp out stiff components
as $\abs{\lambda} \to \infty$.

The following stability properties have been obtained in
\cite{lundquist2014sbp,boom2015high}.
\begin{theorem}
\label{thm:SBP-SAT-stability}
  Consider a first derivative SBP operator $D$.
  If $D + M^{-1} \vec{t}_L \vec{t}_L^T$ is invertible,
  then the SBP-SAT scheme \eqref{eq:SBP-SAT} is both $A$ and $L$ stable.
\end{theorem}

\begin{corollary}
  The SBP-SAT scheme \eqref{eq:SBP-SAT} is both $A$ and $L$ stable
  if $D$ is a nullspace consistent SBP operator.
\end{corollary}
\begin{proof}
  This result follows immediately from Theorem~\ref{thm:nullspace-consistent}
  and Theorem~\ref{thm:SBP-SAT-stability}.
\end{proof}

\section{The New Schemes}
\label{sec:new-schemes}

The idea behind the novel SBP time integration scheme introduced in the
following is to mimic the reformulation of the ODE \eqref{eq:ode} as
an integral equation
\begin{equation}
\label{eq:ode-as-integral}
  u(t) = u_0 + \int_0^t f(\tau, u(\tau)) \dif \tau.
\end{equation}
Taking the time derivative on both sides yields $u'(t) = f(t, u(t))$.
The initial condition $u(0) = u_0$ is satisfied because
$\int_0^0 f(\tau, u(\tau)) \dif \tau = 0$.
Hence, the solution $u$ of \eqref{eq:ode} can be written implicitly
as the solution of the integral equation \eqref{eq:ode-as-integral}.
Note that the integral operator $\int_0^t \cdot \dif \tau$ is the inverse
of the derivative operator $\od{}{t}$ with a vanishing initial condition
at $t = 0$. Hence, a discrete inverse (an integral operator)
of the discrete derivative operator $D$ with a vanishing initial condition
will be our target.

\begin{definition}
  In the space of discrete grid functions, the scalar product induced by
  $M$ is used throughout this article. The adjoint operators with respect
  to this scalar product will be denoted by $\cdot^*$, i.e.\
  $D^* = M^{-1} D^T M$.
  The adjoint of a discrete grid function $\vec{u}$ is denoted by
  $\vec{u}^* = \vec{u}^T M$.
\end{definition}
By definition, the adjoint operator $D^*$ of $D$ satisfies
\begin{equation}
  \scp{\vec{u}}{D^* \vec{v}}_M
  =
  \vec{u}^T M D^* \vec{v}
  =
  \vec{u}^T D^T M \vec{v}
  =
  (D \vec{u})^T M \vec{v}
  =
  \scp{D \vec{u}}{\vec{v}}
\end{equation}
for all grid functions $\vec{u}, \vec{v}$. The adjoint $\vec{u}^*$
is a discrete representation of the inverse Riesz map applied to a
grid function $\vec{u}$ \cite[Theorem~9.18]{roman2008advanced} and
satisfies
\begin{equation}
  \vec{u}^* \vec{v}
  =
  \vec{u}^T M \vec{v}
  =
  \scp{\vec{u}}{\vec{v}}_M.
\end{equation}

The following lemma and definition were introduced in \cite{ranocha2020discrete}.
\begin{lemma}
  For a nullspace consistent first derivative SBP operator $D$,
  $\dim \kernel D^* = 1$.
\end{lemma}

\begin{definition}
  A fixed but arbitrarily chosen basis vector of $\kernel D^*$ for a nullspace
  consistent SBP operator $D$ is denoted as $\osc$.
\end{definition}
The name $\osc$ is intended to remind the reader of (grid) oscillations,
since the kernel of $D^*$ is orthogonal to the image of $D$
\cite[Theorem~10.3]{roman2008advanced} which contains all sufficiently
resolved functions. Several examples are given in \cite{ranocha2020discrete}.
To prove $\kernel D^* \perp \image D$, choose any $D \vec{u} \in \image D$
and $\vec{v} \in \kernel D^*$ and compute
\begin{equation}
  \scp{D \vec{u}}{\vec{v}}_M
  =
  \scp{\vec{u}}{D^* \vec{v}}_M
  =
  \scp{\vec{u}}{\vec{0}}_M
  =
  0.
\end{equation}

\begin{example}
\label{ex:Lobatto-D-D*-o}
  Consider the SBP operator of order $p = 1$ defined by the $p+1 = 2$
  Lobatto-Legendre nodes $\tau_1 = 0$ and $\tau_2 = T$ in $[0, T]$.
  Then,
  \begin{equation}
  \begin{aligned}
    D &= \frac{1}{T} \begin{pmatrix} -1 & 1 \\ -1 & 1 \end{pmatrix},
    &
    M &= \frac{T}{2} \begin{pmatrix} 1 & 0 \\ 0 & 1 \end{pmatrix},
    &
    \vec{t}_L &= \begin{pmatrix} 1 \\ 0 \end{pmatrix},
    &
    \vec{t}_R &= \begin{pmatrix} 0 \\ 1 \end{pmatrix}.
  \end{aligned}
  \end{equation}
  Therefore,
  \begin{equation}
    D^*
    =
    M^{-1} D^T M
    =
    \frac{1}{T} \begin{pmatrix} -1 & -1 \\ 1 & 1 \end{pmatrix}
  \end{equation}
  and $\kernel D^* = \spann \set{\osc}$, where $\osc = (-1, 1)^T$.
  Here, $\osc$ represents the highest resolvable grid oscillation
  on $[\tau_1, \tau_2]$ and $\osc$ is orthogonal to $\image D$, since
  $\osc^T M D = \vec{0}^T$.
\end{example}

The following technique has been used in \cite{ranocha2020discrete} to analyze
properties of SBP operators in space. Here, it will be used to create new SBP
schemes in time.
Consider a nullspace consistent first derivative SBP operator $D$ on the interval
$[0, T]$ using $s$ grid points and the corresponding subspaces
\begin{equation}
  V_0 = \set{\vec{u} \in \R^s | \vec{u}(t = 0) = \vec{t}_L^T \vec{u} = 0 },
  \quad
  V_1 = \set{\vec{u} \in \R^s | \exists \vec{v} \in \R^s \colon \vec{u} = D \vec{v}}.
\end{equation}
Here and in the following, $\vec{u}(t = 0)$ denotes the value of the
discrete function $\vec{u}$ at the initial time $t = 0$. For example,
$\vec{u}(t = 0) = \vec{t}_L^T \vec{u} = \vec{u}^{(1)}$ is the first
coefficient of $\vec{u}$ if $\tau_1 = 0$ and $\vec{t}_L = (1, 0, \dots, 0)^T$.

$V_0$ is the vector space of all grid functions which vanish at the left
boundary point, i.e.\ $V_0 = \kernel \vec{t}_L^T$.
$V_1$ is the vector space of all grid functions which can be represented
as derivatives of other grid functions, i.e.\ $V_1 = \image D$ is the
image of~$D$.

\begin{remark}
  From this point in the paper, $D$ denotes a nullspace consistent
  first derivative SBP operator.
\end{remark}

\begin{lemma}
\label{lem:D-V0-V1-bijective}
  The mapping $D\colon V_0 \to V_1$ is bijective,
  i.e.\ one-to-one and onto, and hence invertible.
\end{lemma}
\begin{proof}
  Given $\vec{u} \in V_1$, there is a $\vec{v} \in \R^s$ such that
  $\vec{u} = D \vec{v}$. Hence,
  \begin{equation}
  \label{eq:D-V0-V1-bijective-proof1}
    D (\vec{v} - (\vec{t}_L^T \vec{v}) \vec{1})
    =
    D \vec{v} - (\vec{t}_L^T \vec{v}) D \vec{1}
    =
    D \vec{v}
    =
    \vec{u}
  \end{equation}
  and
  \begin{equation}
  \label{eq:D-V0-V1-bijective-proof2}
    \vec{t}_L^T (\vec{v} - (\vec{t}_L^T \vec{v}) \vec{1})
    =
    \vec{t}_L^T \vec{v} - \vec{t}_L^T \vec{v}
    =
    0,
  \end{equation}
  since $\vec{t}_L^T \vec{v}$ is a scalar. Hence,
  \eqref{eq:D-V0-V1-bijective-proof2} implies that
  $\vec{v} - (\vec{t}_L^T \vec{v}) \vec{1} \in V_0$.
  Moreover, \eqref{eq:D-V0-V1-bijective-proof1} shows that
  an arbitrary $\vec{u} \in V_1$ can be written as the image
  of a vector in $V_0$ under $D$.
  Therefore, $D\colon V_0 \to V_1$ is surjective (i.e.\ onto).

  To prove that $D$ is injective (i.e.\ one-to-one), consider an arbitrary
  $\vec{u} \in V_1$ and assume there are $\vec{v}, \vec{w} \in V_0$
  such that $D \vec{v} = \vec{u} = D \vec{w}$. Then,
  $D (\vec{v} - \vec{w}) = \vec{0}$. Because of nullspace consistency,
  $\vec{v} - \vec{w} = \alpha \vec{1}$ for a scalar $\alpha$.
  Since $\vec{v}, \vec{w} \in V_0 = \kernel \vec{t}_L^T$,
  \begin{equation}
    0
    =
    \vec{t}_L^T (\vec{v} - \vec{w})
    =
    \alpha \vec{t}_L^T \vec{1}
    =
    \alpha.
  \end{equation}
  Thus, $\vec{v} = \vec{w}$.
\end{proof}

\begin{remark}
\label{rem:V_0}
  $V_0$ is isomorphic to the quotient space $\R^s / \kernel D$, since $D$ is
  nullspace consistent. Hence, Lemma~\ref{lem:D-V0-V1-bijective} basically
  states that $D$ is a bijective mapping from $V_0 \cong \R^s / \kernel D$
  to $V_1 = \image D$.
\end{remark}

\begin{definition}
  The inverse operator of $D\colon V_0 \to V_1$ is denoted as
  $\Dinv\colon V_1 \to V_0$.
\end{definition}
The inverse operator $\Dinv$ is a discrete integral operator such that
$\Dinv \vec{v} \approx \int_0^t v(\tau) \dif \tau$. In general, there is a
one-parameter family of integral operators given by $\int_{t_0}^t v(\tau) \dif \tau$.
Here, we chose the one with $t_0 = 0$ to be consistent with
\eqref{eq:ode-as-integral}.

\begin{example}
\label{ex:Lobatto-V0-V1-Dinv}
  Continuing Example~\ref{ex:Lobatto-D-D*-o}, the vector spaces $V_0$
  and $V_1$ are
  \begin{equation}
    V_0
    =
    \set{ \vec{u} = \alpha \begin{pmatrix} 0 \\ 1 \end{pmatrix} | \alpha \in \R},
    \qquad
    V_1
    =
    \set{ \vec{u} = \beta \begin{pmatrix} 1 \\ 1 \end{pmatrix} | \beta \in \R}.
  \end{equation}
  This can be seen as follows.
  For $V_0 = \kernel \vec{t}_L^T$, using $\vec{t}_L^T = (1, 0)$ implies
  that the first component of $\vec{u} \in V_0$ is zero and that the
  second one can be chosen arbitrarily.
  For $V_1 = \image D$, note that both rows of $D$ are identical.
  Hence, every $\vec{u} = D \vec{v} \in V_1$ must have the same
  first and second component.

  At the level of $\R^2$, the inverse $\Dinv$ of $D$ can be represented as
  \begin{equation}
  \label{eq:ex-Lobatto-V0-V1-Dinv-Dinv1}
    \Dinv
    =
    T
    \begin{pmatrix}
      0 & 0 \\
      0 & 1
    \end{pmatrix}.
  \end{equation}
  Indeed, if $\vec{u} = \begin{pmatrix} 0 \\ \vec{u}_2 \end{pmatrix} \in V_0$,
  then
  \begin{equation}
    \Dinv D \vec{u}
    =
    T
    \begin{pmatrix}
      0 & 0 \\
      0 & 1
    \end{pmatrix}
    \frac{1}{T}
    \begin{pmatrix}
      -1 & 1 \\
      -1 & 1
    \end{pmatrix}
    \begin{pmatrix}
      0 \\
      \vec{u}_2
    \end{pmatrix}
    =
    \begin{pmatrix}
      0 & 0 \\
      0 & 1
    \end{pmatrix}
    \begin{pmatrix}
      \vec{u}_2 \\
      \vec{u}_2
    \end{pmatrix}
    =
    \begin{pmatrix}
      0 \\
      \vec{u}_2
    \end{pmatrix}
    =
    \vec{u}.
  \end{equation}
  Similarly, if $\vec{u} = \vec{u}_1 \begin{pmatrix} 1 \\ 1 \end{pmatrix} \in V_1$,
  then
  \begin{equation}
    D \Dinv \vec{u}
    =
    \frac{1}{T}
    \begin{pmatrix}
      -1 & 1 \\
      -1 & 1
    \end{pmatrix}
    T
    \begin{pmatrix}
      0 & 0 \\
      0 & 1
    \end{pmatrix}
    \begin{pmatrix}
      \vec{u}_1 \\
      \vec{u}_1
    \end{pmatrix}
    =
    \begin{pmatrix}
      -1 & 1 \\
      -1 & 1
    \end{pmatrix}
    \begin{pmatrix}
      0 \\
      \vec{u}_1
    \end{pmatrix}
    =
    \begin{pmatrix} \vec{u}_1 \\ \vec{u}_1 \end{pmatrix}
    =
    \vec{u}.
  \end{equation}
  Hence, $\Dinv D = \id_{V_0}$ and $D \Dinv = \id_{V_1}$, where
  $\id_{V_i}$ is the identity on $V_i$.

  Note that the matrix representation of $\Dinv$ at the level of $\R^s$ is
  not unique since $\Dinv$ is only defined on $V_1 = \kernel \osc^* = \image D$.
  In general, a linear mapping from $\R^s$ to $\R^s$ is determined uniquely by
  $s^2$ real parameters (the entries of the corresponding matrix representation).
  Since $\Dinv$ is defined as a mapping between the $(s-1)$-dimensional spaces
  $V_1$ and $V_0$, it is given by $(s-1)^2$ parameters. Requiring that $\Dinv$
  maps to $V_0$ yields $s$ additional constraints $\vec{t}_L^T \Dinv = \vec{0}^T$.
  Hence, $s - 1$ degrees of freedom remain for any matrix representation of
  $\Dinv$ at the level of $\R^s$. Indeed, adding $\vec{v} \osc^*$ to any matrix
  representation of $\Dinv$ in $\R^s$ results in another valid representation
  if $\vec{v} \in V_0 = \kernel \vec{t}_L^T$.
  In this example, another valid representation of $\Dinv$ at the level of $\R^2$ is
  \begin{equation}
  \label{eq:ex-Lobatto-V0-V1-Dinv-Dinv2}
    \Dinv = T
    \begin{pmatrix}
      0 & 0 \\
      1 & 0
    \end{pmatrix},
  \end{equation}
  which still satisfies $\Dinv D = \id_{V_0}$, $D \Dinv = \id_{V_1}$, and
  yields the same results as the previous matrix representation when applied
  to any $\vec{v} \in V_1$.
\end{example}

Now, we have introduced the inverse $\Dinv$ of $D\colon V_0 \to V_1$,
which is a discrete integral operator $\Dinv\colon V_1 \to V_0$. However,
the integral operator $\Dinv$ is only defined for elements
of the space $V_1 = \image D$.
Hence, one has to make sure that a generic right hand side vector $\vec{f}$
is in the range of the derivative operator $D$ in order to apply the
inverse $\Dinv$. To guarantee this, components in the direction
of grid oscillations $\osc$ must be removed. For this, the discrete
projection/filter operator
\begin{equation}
\label{eq:F}
  \F = \I - \frac{\osc \osc^*}{\norm{\osc}_M^2}
\end{equation}
will be used.
\begin{lemma}
\label{lem:F}
  The projection/filter operator $\F$ defined in \eqref{eq:F} is an orthogonal
  projection onto the range of $D$, i.e.\ onto
  $V_1 = \image D = (\kernel D^*)^\perp$.
  It is symmetric and positive semidefinite with respect to the scalar
  product induced by $M$.
\end{lemma}
\begin{proof}
  Clearly, $\frac{\osc \osc^*}{\norm{\osc}_M^2}$ is the usual
  orthogonal projection onto $\spann\{ \osc \}$
  \cite[Theorems~9.14 and~9.15]{roman2008advanced}.
  Hence, $\F$ is the orthogonal
  projection onto the orthogonal complement
  $\spann\{ \osc \}^\perp = (\kernel D^*)^\perp = \image D = V_1$.
  In particular, for a (real or complex valued) discrete grid
  function $\vec{u}$,
  \begin{equation}
    \scp{\vec{u}}{\F \vec{u}}_M
    =
    \scp{\vec{u}}{\vec{u}}_M - \frac{|\scp{\vec{u}}{\osc}_M|^2}{\norm{\osc}_M^2}
    \geq
    0
  \end{equation}
  because of the Cauchy-Schwarz inequality
  \cite[Theorem~9.3]{roman2008advanced}.
\end{proof}

Now, all ingredients to mimic the integral equation \eqref{eq:ode-as-integral}
have been provided. Applying at first the discrete projection operator $\F$
and second the discrete integral operator $J$ to a generic right hand side
$\vec{f}$ results in $\Dinv \F \vec{f}$, which is a discrete analog of the
integral $\int_0^t f \dif \tau$.
Additionally, the initial condition has to be imposed, which is done by
adding the constant initial value as $u_0 \otimes \vec{1}$.
Putting it all together, a new class of SBP schemes mimicking the integral
equation \eqref{eq:ode-as-integral} discretely is proposed as
\begin{equation}
\label{eq:SBP-proj}
  \vec{u}
  =
  u_0 \otimes \vec{1}
  + \Dinv \F \vec{f},
  \quad
  u_+ = \vec{t}_R^T \vec{u}.
\end{equation}
For a scalar ODE \eqref{eq:ode}, the first term on the right-hand side of
the proposed scheme~\eqref{eq:SBP-proj} is $u_0 \otimes \vec{1} = (u_0, \dots,
u_0)^T \in \R^s$.
Note that \eqref{eq:SBP-proj} is an implicit scheme since
$f = f(t, u)$.

\begin{example}
\label{ex:Lobatto-SBP-proj}
  Continuing Examples~\ref{ex:Lobatto-D-D*-o} and~\ref{ex:Lobatto-V0-V1-Dinv},
  the adjoint of $\osc$ is
  \begin{equation}
    \osc^*
    =
    \osc^T M
    =
    (-1, 1)
    \frac{T}{2}
    \begin{pmatrix}
      1 & 0 \\
      0 & 1
    \end{pmatrix}
    =
    \frac{T}{2}
    (-1, 1).
  \end{equation}
  Hence, ${\norm{\osc}_M^2} = \osc^* \osc = T$,
  and the projection/filter operator \eqref{eq:F} is
  \begin{equation}
    \F
    =
    \I - \frac{\osc \osc^*}{\norm{\osc}_M^2}
    =
    \begin{pmatrix}
      1 & 0 \\
      0 & 1
    \end{pmatrix}
    -
    \frac{1}{2}
    \begin{pmatrix}
      1 & -1 \\
      -1 & 1
    \end{pmatrix}
    =
    \frac{1}{2}
    \begin{pmatrix}
      1 & 1 \\
      1 & 1
    \end{pmatrix}.
  \end{equation}
  Thus, $\F$ is a smoothing filter operator that removes the highest
  grid oscillations and maps a grid function into the image of the
  derivative operator $D$. Hence, the inverse $\Dinv$, the discrete
  integral operator, can be applied after $\F$, resulting in
  \begin{equation}
  \label{eq:ex-Lobatto-SBP-proj-DinvF}
    \Dinv \F
    =
    \Dinv \left( \I - \frac{\osc \osc^*}{\norm{\osc}_M^2} \right)
    =
    T
    \begin{pmatrix}
      0 & 0 \\
      0 & 1
    \end{pmatrix}
    \frac{1}{2}
    \begin{pmatrix}
      1 & 1 \\
      1 & 1
    \end{pmatrix}
    =
    \frac{T}{2}
    \begin{pmatrix}
      0 & 0 \\
      1 & 1
    \end{pmatrix}.
  \end{equation}
  Finally, for an arbitrary $\vec{u} \in \R^2$,
  \begin{equation}
    \Dinv \F \vec{u}
    =
    \frac{T}{2}
    \begin{pmatrix}
      0 & 0 \\
      1 & 1
    \end{pmatrix}
    \begin{pmatrix}
      \vec{u}_1 \\
      \vec{u}_2
    \end{pmatrix}
    =
    \frac{T}{2}
    \begin{pmatrix}
      0 \\
      \vec{u}_1 + \vec{u}_2
    \end{pmatrix}
    \in V_0.
  \end{equation}
  A more involved example of the development presented here is given in
  \autoref{sec:example-gauss}.
\end{example}

\subsection{Summarizing the Development}
\label{sec:summarizing-SBP-proj}

As stated earlier, the SBP time integration scheme \eqref{eq:SBP-proj}
mimics the integral reformulation \eqref{eq:ode-as-integral}
of the ODE \eqref{eq:ode}.
Instead of using $\Dinv$ as discrete analog of the integral operator
$\int_0^t \cdot \dif \tau$ directly, the projection/filter operator $\F$ defined in
\eqref{eq:F} must be applied first in order to guarantee that the generic
vector $\vec{f}$ is in the image of $D$.
Finally, the initial condition is imposed strongly.

Note that
\begin{equation}
\label{eq:tL-u}
  \vec{t}_L^T \vec{u}
  =
  \vec{t}_L^T (u_0 \otimes \vec{1})
  + \vec{t}_L^T \Dinv \left( \I - \frac{\osc \osc^*}{\norm{\osc}_M^2} \right) \vec{f}
  =
  u_0.
\end{equation}
The second summand vanishes because $\Dinv$ returns a
vanishing value at $t = 0$, i.e.\ $\vec{t}_L^T \Dinv = \vec{0}^T$,
since $\Dinv\colon V_1 \to V_0$ maps onto $V_0 = \kernel \vec{t}_L^T$.

Note that the projection/filter operator $\F$ is required in \eqref{eq:SBP-proj},
since the discrete integral operator $\Dinv$ only operates on objects in $V_1$.
Note that the matrix representations of $\Dinv$ in $\R^s$ given in the examples
above are constructed such that they should be applied only to vectors
$\vec{v} \in V_1$. As explained
in Example~\ref{ex:Lobatto-V0-V1-Dinv}, the matrix representation of $\Dinv$ in
$\R^s$ is not unique. Thus, choosing any of these representations without
applying the filter/projection operator $\F$ would result in undefined/unpredictable
behavior. The projection/filter operator $\F$ is necessary to make
\eqref{eq:SBP-proj} well-defined.
Indeed, the product $\Dinv \F$ is well-defined, i.e.\ it is the same for any
matrix representation of $\Dinv$, since $\F$ maps to $V_1 = \image D$ and the
action of $\Dinv$ is defined uniquely on this space. In particular, $\Dinv \F$
itself is a valid matrix representation of $\Dinv$ in $\R^{s}$.
As an example, $\Dinv \F$ in \eqref{eq:ex-Lobatto-SBP-proj-DinvF} is a linear
combination of the possible representations \eqref{eq:ex-Lobatto-V0-V1-Dinv-Dinv1}
and \eqref{eq:ex-Lobatto-V0-V1-Dinv-Dinv2} of $\Dinv$ in $\R^{s}$ and thus also
a representation of $\Dinv$ in $R^{s}$.


Another argument for the necessity of the filter/projection operator $\F$ can
be derived using the following result.
\begin{lemma}
\label{lem:SBP-proj-Du}
  Let $D$ be a nullspace consistent first derivative SBP operator.
  Then, $D \Dinv \F = F$ and the solution $\vec{u}$ of
  \eqref{eq:SBP-proj} satisfies
  \begin{equation}
  \label{eq:SBP-proj-Du}
    D \vec{u}
    =
    \F \vec{f}.
  \end{equation}
\end{lemma}
Before proving Lemma~\ref{lem:SBP-proj-Du}, we discuss its meaning here.
In general, it is not possible to find a solution $\vec{u}$ of $D \vec{u} = \vec{f}$
for an arbitrary right-hand side $\vec{f}$, since $D$ is not invertible on
$\R^s$. Multiplying $\vec{f}$ by the orthogonal projection operator $\F$ ensures
that the new right-hand side $\F \vec{f}$ of \eqref{eq:SBP-proj-Du} is in the
image of $D$ and hence that \eqref{eq:SBP-proj-Du} can be solved for any given
$\vec{f}$. This projection $\F$ onto $V_1 = \image D$ is necessary in the discrete
case because of the finite dimensions.

\begin{proof}[Proof of Lemma~\ref{lem:SBP-proj-Du}]
  Taking the discrete derivative on both sides of
  \eqref{eq:SBP-proj} results in
  \begin{equation}
    D \vec{u}
    =
    D \Dinv \F \vec{f},
  \end{equation}
  since $D (u_0 \otimes \vec{1}) = u_0 \otimes (D \vec{1}) = \vec{0}$.
  Hence, \eqref{eq:SBP-proj-Du} holds if $D \Dinv \F = F$.
  To show $D \Dinv \F = F$, it suffices to show $D \Dinv \F \vec{f} = F \vec{f}$
  for arbitrary $\vec{f}$.  Write $\vec{f}$ as
  $\vec{f} = D \vec{v} + \alpha \osc$,
  where $\vec{t}_L^T \vec{v} = 0$ and $\alpha \in \R$. This is always possible
  since $D$ is nullspace consistent. Then,
  \begin{equation}
    D \Dinv \F \vec{f}
    =
    D \Dinv \F D \vec{v} + \alpha D \Dinv \F \osc
    =
    D \Dinv D \vec{v}
    =
    D \vec{v},
  \end{equation}
  where we used
  $\F D = D$,
  $\F \osc = \vec{0}$,
  $\Dinv D = \id_{V_0}$, and
  $\vec{v} \in V_0 = \kernel \vec{t}_L^T$.
  Using again
  $\F D = D$ and
  $\F \osc = \vec{0}$, we get
  \begin{equation}
    D \vec{v}
    =
    \F D \vec{v} + \vec{0}
    =
    \F (D \vec{v} + \alpha \osc)
    =
    \F \vec{f}.
  \end{equation}
  Hence, $D \Dinv \F \vec{f} = F \vec{f}$.
\end{proof}

\begin{remark}
  In the context of SBP operators, two essentially different interpretations
  of integrals arise. Firstly, the integral $\int_0^T \cdot \dif \tau$ gives
  the $L^2$ scalar product, approximated by the mass matrix $M$ which maps
  discrete functions to scalar values.
  Secondly, the integral $\int_0^t \cdot \dif \tau$ is the inverse of the
  derivative with vanishing values at $t = 0$. This operator is discretized
  as $\Dinv$ on its domain of definition $V_1 = \image D$ and maps a discrete
  grid function in $V_1 = \image D$ to a discrete grid function in
  $V_0 = \kernel \vec{t}_L^T$.
\end{remark}

\subsection{Linear Stability}

In this section, linear stability properties of the new scheme
\eqref{eq:SBP-proj} are established.
\begin{theorem}
\label{thm:A-stability}
  For nullspace consistent SBP operators, the scheme \eqref{eq:SBP-proj} is
  $A$ stable.
\end{theorem}
\begin{proof}
  For the scalar linear ODE \eqref{eq:scalar-test-ode} with $\Re \lambda \leq 0$,
  the energy method will be applied to the scheme \eqref{eq:SBP-proj}.
  We write $\overline{\cdot}$ to denote the complex conjugate.
  Using $\vec{t}_L^T \vec{u} = u_0$ from \eqref{eq:tL-u} and
  $u_+ = \vec{t}_R^T \vec{u}$ from the definition of the scheme \eqref{eq:SBP-proj},
  the difference of the energy at the final and initial time is
  \begin{equation}
  \label{eq-thm:A-stability}
  \begin{aligned}
    \abs{u_+}^2 - \abs{u_0}^2
    &=
    \abs{\vec{t}_R^T \vec{u}}^2 - \abs{\vec{t}_L^T \vec{u}}^2
    =
    \overline{\vec{u}}^T \vec{t}_R \vec{t}_R^T \vec{u}
    - \overline{\vec{u}}^T \vec{t}_L \vec{t}_L^T \vec{u}
    \\
    &=
    \overline{\vec{u}}^T \bigl( M D + (M D)^T \bigr) \vec{u},
  \end{aligned}
  \end{equation}
  where the SBP property \eqref{eq:SBP} has been used in the last
  equality. As shown in Lemma~\ref{lem:SBP-proj-Du}, the scheme \eqref{eq:SBP-proj}
  yields $D \vec{u} = \F \vec{f}$. For the scalar linear ODE
  \eqref{eq:scalar-test-ode}, $\vec{f} = \lambda \vec{u}$. Hence, we can replace
  $D \vec{u}$ by $\F \vec{f} = \lambda \F \vec{u}$ in \eqref{eq-thm:A-stability},
  resulting in
  \begin{equation}
  \begin{aligned}
    \abs{u_+}^2 - \abs{u_0}^2
    &=
    2 \Re\bigl( \overline{\vec{u}}^T M D \vec{u} \bigr)
    \\
    &=
    2 \Re\bigl( \lambda \overline{\vec{u}}^T M \F \vec{u} \bigr)
    =
    2 \underbrace{\Re(\lambda)}_{\leq 0}
    \underbrace{\overline{\vec{u}}^T M \F \vec{u}}_{\geq 0}
    \leq
    0.
  \end{aligned}
  \end{equation}
  The second factor is non-negative because $\F$ is positive semidefinite
  with respect to the scalar product induced by $M$, cf.
  Lemma~\ref{lem:F}.
  Therefore, $\abs{u_+}^2 \leq \abs{u_0}^2$, implying that the scheme
  is $A$ stable.
\end{proof}

In general, the novel SBP scheme \eqref{eq:SBP-proj} is not
$L$ stable, cf. Remark~\ref{rem:L-B-stability}.

\subsection{Characterization as a Runge-Kutta Method}

Unsurprisingly, the new method \eqref{eq:SBP-proj} can be characterized as
a Runge-Kutta method.
\begin{theorem}
\label{thm:RK-characterization}
  For nullspace consistent SBP operators that are at least first order
  accurate, the method \eqref{eq:SBP-proj} is
  a Runge-Kutta method with Butcher coefficients
  \begin{equation}
  \label{eq:RK-characterization}
    A
    =
    \frac{1}{T} \Dinv \F
    =
    \frac{1}{T} \Dinv \left( \I - \frac{\osc \osc^*}{\norm{\osc}_M^2} \right),
    \quad
    b = \frac{1}{T} M \vec{1},
    \quad
    c = \frac{1}{T} (\tau_1, \dots, \tau_s)^T.
  \end{equation}
\end{theorem}
\begin{proof}
  First, note that one step \eqref{eq:RK-step} from zero to $T$ of
  a Runge-Kutta method with coefficients $A, b, c$ can be written as
  \begin{equation}
  \label{eq:RK-step-tensor-product}
    \vec{u} = u_0 \otimes \vec{1} + T A \vec{f},
    \quad
    u_+ = u_0 + T b^T \vec{f},
  \end{equation}
  where the right hand side vector $\vec{f}$ is given by
  $\vec{f}_i = f(c_i T, \vec{u}_i)$.
  Comparing this expression with
  \begin{equation}
  \tag{\ref{eq:SBP-proj}}
    \vec{u}
    =
    u_0 \otimes \vec{1}
    + \Dinv \F \vec{f},
    \quad
    u_+ = \vec{t}_R^T \vec{u},
  \end{equation}
  where the right hand side is given by
  $\vec{f}_i = f(\tau_i, \vec{u}_i)$,
  the form of $A$ and $c$ is immediately clear.
  The new value $u_+$ of the new SBP method \eqref{eq:SBP-proj} is
  \begin{equation}
    u_+
    =
    \vec{t}_R^T \vec{u}
    =
    \vec{1}^T \vec{t}_R \vec{t}_R^T \vec{u}.
  \end{equation}
  Using the SBP property \eqref{eq:SBP} and $D \vec{1} = \vec{0}$,
  \begin{equation}
    u_+
    =
    \vec{1}^T \vec{t}_L \vec{t}_L^T \vec{u}
    + \vec{1}^T M D \vec{u}
    + \vec{1}^T D^T M \vec{u}
    =
    \vec{t}_L^T \vec{u} + \vec{1}^T M D \vec{u}.
  \end{equation}
  Inserting $\vec{t}_L^T \vec{u} = u_0$ and $D \vec{u}$ from
  \eqref{eq:SBP-proj-Du} results in
  \begin{equation}
    u_+
    =
    u_0 + \vec{1}^T M \left( \I - \frac{\osc \osc^*}{\norm{\osc}_M^2} \right) \vec{f}
    =
    u_0 + \vec{1}^T M \vec{f} - \frac{\scp{\vec{1}}{\osc}_M \scp{\osc}{\vec{f}}_M}{\norm{\osc}_M^2}.
  \end{equation}
  Because of $\vec{1} \in \image D \perp \kernel D^* \ni \osc$, we have
  $\scp{\vec{1}}{\osc}_M = 0$. Hence,
  \begin{equation}
    u_+
    =
    u_0 + \vec{1}^T M \vec{f}.
  \end{equation}
  Comparing this expression with \eqref{eq:RK-step-tensor-product}
  yields the final assertion for $b$.
\end{proof}

\begin{lemma}
\label{lem:first-row}
  For nullspace consistent SBP operators, the first row of the Butcher
  coefficient matrix $A$ in \eqref{eq:RK-characterization} of the method
  \eqref{eq:SBP-proj} is zero if $\vec{t}_L = (1, 0, \dots, 0)^T$.
\end{lemma}
\begin{proof}
  By definition, $\Dinv$ yields a vector with vanishing initial
  condition at the left endpoint, i.e.\ $\vec{t}_L^T \Dinv = \vec{0}^T$.
  Because of $\vec{t}_L = (1, 0, \dots, 0)^T$, we have
  $\vec{0}^T = \vec{t}_L^T \Dinv = \Dinv[1,:]$, which is the first
  row of $\Dinv$, where a notation as in Julia \cite{bezanson2017julia}
  has been used.
\end{proof}

\subsection{Operator Construction}
\label{sec:operator-construction}

To implement the
SBP scheme \eqref{eq:SBP-proj}, the product $\Dinv \F$ has to be computed,
which is (except for a scaling by $T^{-1}$) the matrix $A$ of the corresponding
Runge-Kutta method, cf.\ Theorem~\ref{thm:RK-characterization}.
Since the projection operator $\F$ maps vectors into $V_1 = \image D$, the columns
of $\F$ are in the image of the nullspace consistent SBP derivative
operator $D$. Hence, the matrix equation $D X = \F$ can be solved for $X$,
which is a matrix of the same size as $D$, e.g.\ via a QR factorization,
yielding the least norm solution.
Then, we have to ensure that the columns of $\Dinv \F$ are in
$V_0 = \kernel \vec{t}_L^T$, since $\Dinv$ maps $V_1$ into $V_0$.
This can be achieved by subtracting $\vec{t}_L^T X[:,j]$ from each
column $X[:,j]$ of $X$, $j \in \{1, \dots, s\}$, where a notation
as in Julia \cite{bezanson2017julia} has been used. After this correction,
we have $X = \Dinv \F$. Finally, we need to solve \eqref{eq:SBP-proj} for $\vec{u}$
for each step, using the operator $\Dinv \F$ constructed as described above.

\subsection{Lobatto~IIIA Schemes}

General characterizations of the SBP-SAT scheme \eqref{eq:SBP-SAT} on Radau and
Lobatto nodes as classical collocation Runge-Kutta methods (Radau~IA and IIA or
Lobatto~IIIC, respectively) have been obtained in \cite{ranocha2019some}. A
similar characterization will be obtained in this section.

\begin{theorem}
\label{thm:Lobatto-IIIA}
  If the SBP operator $D$ is given by the nodal polynomial collocation scheme on
  Lobatto-Legendre nodes, the SBP method \eqref{eq:SBP-proj} is the classical
  Lobatto~IIIA method.
\end{theorem}
\begin{proof}
  The Lobatto~IIIA methods are given by the nodes $c$ and weights $b$
  of the Lobatto-Legendre quadrature, just as the SBP method
  \eqref{eq:SBP-proj}. Hence, it remains to prove that the classical condition
  $C(s)$ is satisfied \cite[Section~344]{butcher2016numerical}, where
  \begin{equation}
  \label{eq:C-eta}
    C(\eta)\colon
    \qquad \sum_{j=1}^s a_{i,j} c_j^{q-1} = \frac{1}{q} c_i^q,
    \quad i \in \set{1, \dots, s},\; q \in \set{1, \dots, \eta}.
  \end{equation}
  In other words, all polynomials of degree $\leq p = s - 1$ must be integrated exactly
  by $A$ with vanishing initial value at $t = 0$. By construction of
  $A$, see \eqref{eq:RK-characterization}, this is satisfied for all
  polynomials of degree $\leq p-1$, since the grid oscillations are given by
  $\osc = \vec{\phi}_p$, where $\vec{\phi}_p$ is the Legendre polynomial
  of degree $p$, cf. \cite[Example~3.6]{ranocha2020discrete}.

  Finally, it suffices to check whether $\vec{\phi}_p$ is integrated
  exactly by $A$. The left hand side of \eqref{eq:C-eta} yields
  \begin{equation}
    A \vec{\phi}_p
    =
    \frac{1}{T} \Dinv \left( \I - \frac{\osc \osc^*}{\norm{\osc}_M^2} \right) \vec{\phi}_p
    =
    0,
  \end{equation}
  since $\osc = \vec{\phi}_p$. On the right hand side, transforming
  the time domain to the
  standard interval $[-1,1]$, the analytical integrand of ${\phi}_p$ is
  \begin{equation}
    \int_{-1}^x {\phi}_p(s) \dif s
    =
    \frac{1}{p (p+1)} \int_{-1}^x \partial_s \left[ (s^2 - 1) {\phi}_p'(s) \right] \dif s
    =
    \frac{1}{p (p+1)} (x^2 - 1) {\phi}_p'(x),
  \end{equation}
  since the Legendre polynomials satisfy Legendre's differential equation
  \begin{equation}
    \partial_x \bigl( (1 - x^2) \partial_x {\phi}_p(x) \bigr) + p (p+1) {\phi}_p(x) = 0.
  \end{equation}
  Hence, $\int_{-1}^x {\phi}_p(s) \dif s$ vanishes exactly at the $s = p+1$
  Legendre nodes for polynomials of degree $p$, which are $\pm 1$ and the
  roots of ${\phi}_p'$. Thus, the analytical integral of ${\phi}_p$ vanishes
  at all grid nodes.
\end{proof}

\begin{example}
\label{ex:Lobatto-Abc}
  Continuing Examples~\ref{ex:Lobatto-D-D*-o} and~\ref{ex:Lobatto-SBP-proj},
  the nodes $c = \frac{1}{T} (\tau_1, \tau_2)^T = (0, 1)^T$ are the nodes of
  the Lobatto-Legendre quadrature with two nodes in $[0, 1]$.
  Moreover, the corresponding weights are given by
  \begin{equation}
    b = \frac{1}{T} M \vec{1} = \frac{1}{2} \begin{pmatrix} 1 \\ 1 \end{pmatrix}.
  \end{equation}
  Finally, the remaining Butcher coefficients are given by
  \begin{equation}
    A
    =
    \frac{1}{T} \Dinv \F
    =
    \frac{1}{T} \Dinv \left( \I - \frac{\osc \osc^*}{\norm{\osc}_M^2} \right)
    =
    \frac{1}{2}
    \begin{pmatrix}
      0 & 0 \\
      1 & 1
    \end{pmatrix},
  \end{equation}
  which are exactly the coefficients of the Lobatto~IIIA method with
  $s = 2$ stages.
\end{example}

\begin{remark}
\label{rem:L-B-stability}
  Since the Lobatto~IIIA methods are neither $L$ nor $B$ stable,
  the new SBP method \eqref{eq:SBP-proj} is in general not $L$ or $B$
  stable, too.
\end{remark}

\begin{remark}
\label{rem:other-classical-schemes}
  Because of Lemma~\ref{lem:first-row}, the classical Gauss,
  Radau~IA, Lobatto~IIIB, and Lobatto~IIIC methods cannot be
  expressed in the form \eqref{eq:SBP-proj}.
  The classical Radau~I, Radau~II, and Radau~IIA methods are also not
  included in the class \eqref{eq:SBP-proj}. For example,
  for two nodes, these methods have the $A$ matrices
  \begin{equation}
    \begin{pmatrix}
      0 & 0 \\
      \nicefrac{1}{3} & \nicefrac{1}{3}
    \end{pmatrix},
    \qquad
    \begin{pmatrix}
      \nicefrac{1}{3} & 0 \\
      1 & 0
    \end{pmatrix},
    \qquad
    \begin{pmatrix}
      \nicefrac{5}{12} & \nicefrac{-1}{12} \\
      \nicefrac{3}{4} & \nicefrac{1}{4}
    \end{pmatrix}
  \end{equation}
  while the methods \eqref{eq:SBP-proj} on the left and right Radau
  nodes yield the matrices
  \begin{equation}
    \begin{pmatrix}
      0 & 0 \\
      \nicefrac{1}{6} & \nicefrac{1}{2}
    \end{pmatrix},
    \qquad
    \begin{pmatrix}
      \nicefrac{1}{4} & \nicefrac{1}{12} \\
      \nicefrac{3}{4} & \nicefrac{1}{4}
    \end{pmatrix}.
  \end{equation}
\end{remark}

We develop a related SBP time integration scheme that includes the classical
Lobatto~IIIB collocation method in the appendix. Additionally, we mention
why it seems to be difficult to describe Gauss collocation methods in a
general SBP setting, cf. \autoref{sec:LobattoIIIB}.

\subsection{Order of Accuracy}
\label{sec:accuracy}

Next, we establish results on the order of accuracy of the new class of
SBP time integration methods.

\begin{theorem}
\label{thm:SBP-proj-accuracy}
  For nullspace consistent SBP operators that are $p$th order accurate with
  $p \ge 1$, the Runge-Kutta method \eqref{eq:RK-characterization} associated
  to the SBP time integration scheme \eqref{eq:SBP-proj} has an order of accuracy of
  \begin{enumerate}[label=\alph*)]
    \item
    at least $p$ for general mass matrices $M$.

    \item
    at least $2p$ for diagonal mass matrices $M$.
  \end{enumerate}
\end{theorem}
The technical proof of Theorem~\ref{thm:SBP-proj-accuracy} is given in
\autoref{sec:proof-of-thm:SBP-proj-accuracy}.

\begin{remark}
  The result on the order of accuracy given in Theorem~\ref{thm:SBP-proj-accuracy}
  may appear counterintuitive at first when looking from the perspective
  of classical finite difference SBP operators, since diagonal norm matrices
  are usually less accurate in this context. Indeed, finite difference SBP
  operators for the first derivative with a diagonal norm matrix have an order
  of accuracy of $2q$ in the interior and $r \le q$ at the boundaries
  \cite{linders2018order}, where usually $r = q$. In contrast, the corresponding
  dense norm operators have an order of accuracy $2q$ in the interior and $2q-1$
  at the boundaries. Hence, the total order of accuracy $p = q$ for diagonal
  mass matrices is smaller than the order of accuracy $p = 2q-1$ for dense norms.
  However, dense norms are not guaranteed to result in the same high order of
  accuracy when used as a quadrature rule. Thus, the total order of accuracy
  can be smaller even if the pointwise accuracy as a derivative operator is
  higher (which basically corresponds to the \emph{stage order} in the context
  of Runge-Kutta methods).
\end{remark}

\section{Numerical Experiments}
\label{sec:numerical-experiments}

Numerical experiments corresponding to the ones in
\cite{lundquist2014sbp} will be conducted. The novel SBP methods
\eqref{eq:SBP-proj} have been implemented in Julia~v1.5
\cite{bezanson2017julia} and Matplotlib \cite{hunter2007matplotlib}
has been used to generate the plots.
The source code for all numerical examples is available online
\cite{ranocha2020classRepro}.
After computing the operators $A = \frac{1}{T} \Dinv \F$ as described in
Section~\ref{sec:operator-construction} and inserting the right-hand sides $f$
of the ODEs considered in the following into the scheme \eqref{eq:SBP-proj},
the resulting linear systems are solved using the backslash operator in Julia.

Numerical experiments are shown only for new SBP methods \eqref{eq:SBP-proj}
based on finite difference SBP operators and not for methods based on
Lobatto quadrature, since the classical Lobatto~IIIA and IIIB schemes
are already well-known in the literature. The diagonal norm finite difference
SBP operators of \cite{mattsson2004summation} use central finite difference
stencils in the interior of the domain and adapted boundary closures to satisfy
the SBP property \eqref{eq:SBP}.
The Butcher coefficients of some of these methods are given in
\autoref{sec:Butcher-FD}.

\subsection{Non-stiff Problem}

The non-stiff test problem
\begin{equation}
\label{eq:non-stiff}
  u'(t) = - u(t), \quad u(0) = 1,
\end{equation}
with analytical solution $u(t) = \exp(-t)$ is solved in the time interval $[0,1]$
using the SBP method \eqref{eq:SBP-proj} with the diagonal norm operators of
\cite{mattsson2004summation}. The errors of the numerical solutions at the final
time are shown in Figure~\ref{fig:non-stiff}. As can be seen, they converge
with an order of accuracy equal to the interior approximation order
of the diagonal norm operators. For the operator with interior order eight,
the error reaches machine precision for $N = 50$ nodes and does not decrease
further.
These results are comparable to the ones obtained by SBP-SAT schemes in
\cite{lundquist2014sbp} and match the order of accuracy of the corresponding
Runge-Kutta methods guaranteed by Theorem~\ref{thm:SBP-proj-accuracy}.

\begin{figure}
\centering
  \includegraphics[width=0.8\textwidth]{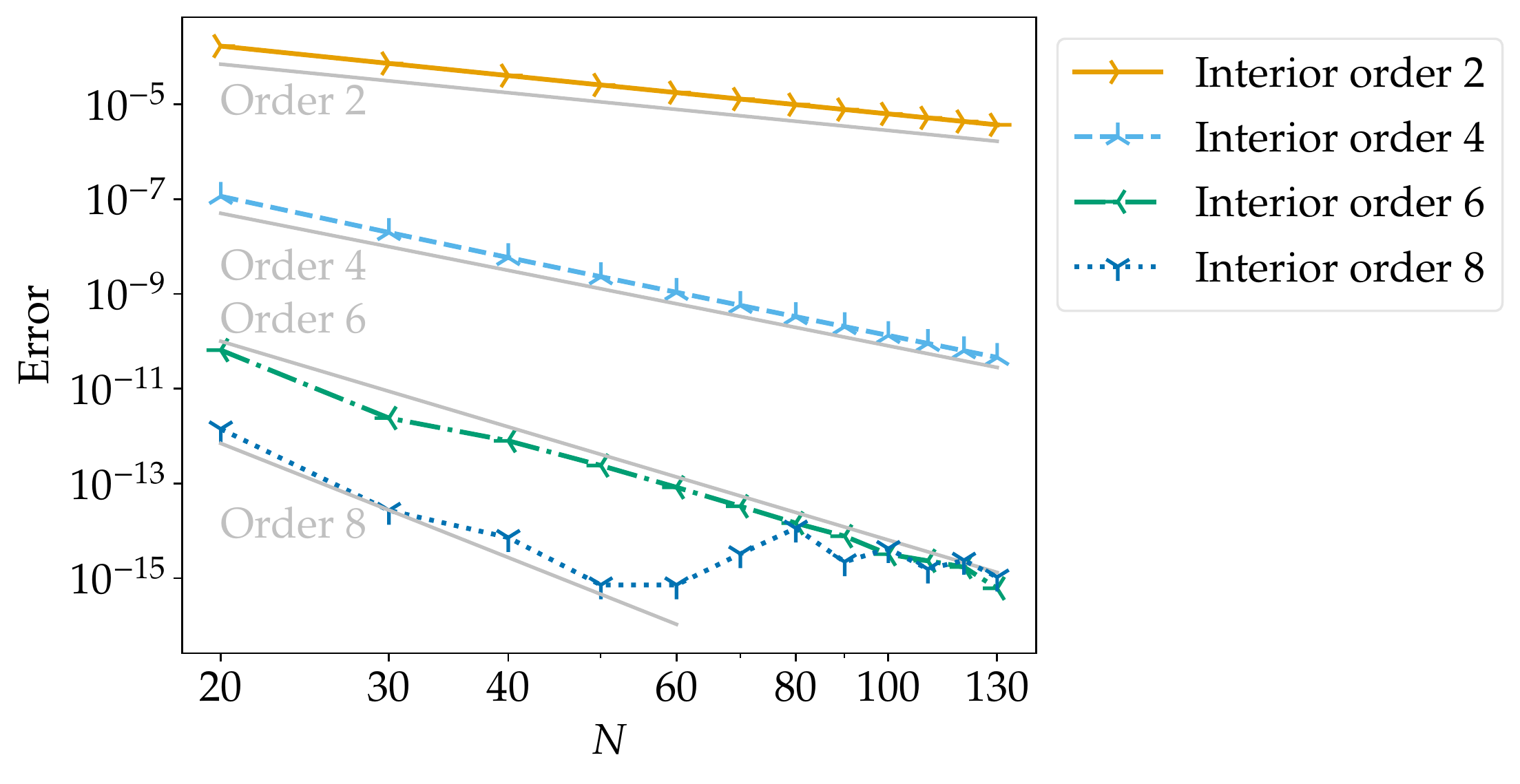}
  \caption{Convergence behavior of the SBP scheme \eqref{eq:SBP-proj} using the
           diagonal norm operators of \cite{mattsson2004summation} for the
           non-stiff test problem \eqref{eq:non-stiff}.}
  \label{fig:non-stiff}
\end{figure}

\subsection{Stiff Problem}

The stiff test problem
\begin{equation}
\label{eq:stiff}
  u'(t) = \lambda \bigl( u(t) - \exp(-t) \bigr) - \exp(-t), \quad u(0) = 1,
\end{equation}
with analytical solution $u(t) = \exp(-t)$ and parameter $\lambda = 1000$ is
solved in the time interval $[0,1]$. The importance of such test problems for
stiff equations has been established in \cite{prothero1974stability}.
Using the diagonal norm operators of \cite{mattsson2004summation} for the
method \eqref{eq:SBP-proj} yields the convergence behavior shown in
Figure~\ref{fig:stiff}.
Again, the results are comparable to the ones obtained by SBP-SAT schemes in
\cite{lundquist2014sbp}. In particular, the order of convergence is reduced
to the approximation order at the boundaries, exactly as for the SBP-SAT
schemes of \cite{nordstrom2013summation}. Such an order reduction for
stiff problems is well-known in the literature on time integration methods,
see \cite{prothero1974stability} and \cite[Chapter~IV.15]{hairer2010solving}.

\begin{figure}
\centering
  \includegraphics[width=0.8\textwidth]{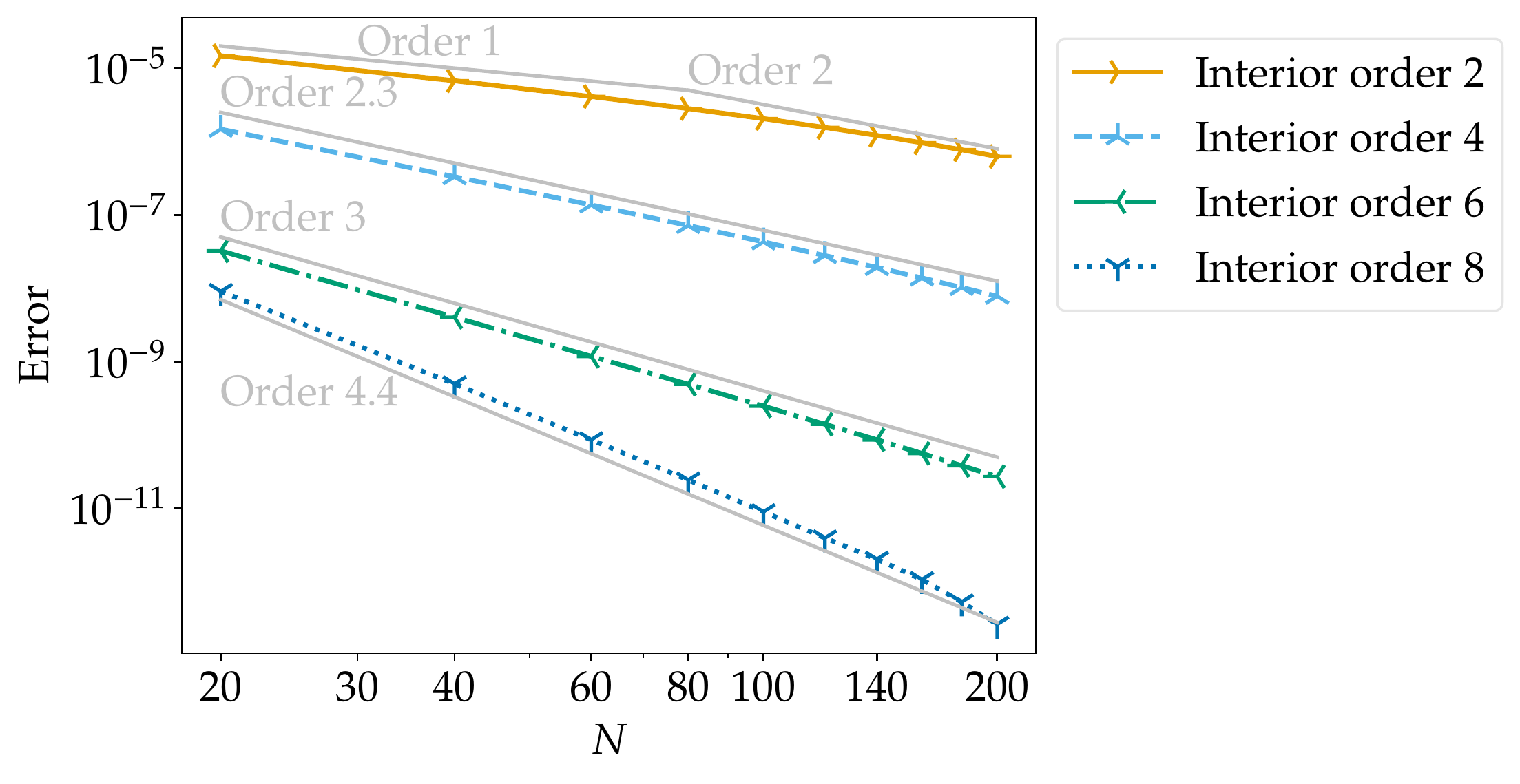}
  \caption{Convergence behavior of the SBP scheme \eqref{eq:SBP-proj} using the
           diagonal norm operators of \cite{mattsson2004summation} for the
           stiff test problem \eqref{eq:stiff}.}
  \label{fig:stiff}
\end{figure}

\section{Summary and Discussion}
\label{sec:summary}

A novel class of $A$ stable summation by parts time integration methods
has been proposed. Instead of using simultaneous approximation terms
to impose the initial condition weakly, the initial condition is imposed
strongly. Similarly to previous SBP time integration methods, the new schemes
can be reformulated as implicit Runge-Kutta methods.

Compared to the SAT approach, some linear and nonlinear stability properties
such as $L$ and $B$ stability are lost in general. On the other hand,
well-known $A$ stable
methods such as the Lobatto~IIIA schemes are included in this new
SBP framework. Additionally, a related SBP time integration method
has been proposed which includes the classical Lobatto~IIIB schemes.

This article provides new insights into the relations of numerical
methods and contributes to the discussion of whether SBP properties are
necessarily involved in numerical schemes for differential equations
which are provably stable.

\subsection{Final Reflections on Obtained Results}

We have concentrated on classical collocation Runge-Kutta methods when looking for
known schemes in the new class of SBP time integration methods, since these
have direct connections to quadrature rules, which are closely connected
to the SBP property \cite{hicken2013summation}. We are not aware of other
classical Runge-Kutta methods that are contained in the new class of SBP
methods proposed and analyzed in this article  besides Lobatto~IIIA and IIIB
schemes. The implicit equations that need to be solved per time step for the
new methods can be easier to solve than the ones occurring in SBP-SAT methods,
e.g.\ since the first stage does not require an implicit solution at all for
some methods (Lemma~\ref{lem:first-row}).
On the other hand, the new classes of methods do not necessarily have the same
kind of nonlinear stability properties as previous SBP-SAT methods. Hence, a
thorough parameter search and comparison of the methods would be necessary for
a detailed comparison, which is beyond the scope of this initial article.

In general, methods constructed using SBP operators often imply certain stability
properties automatically, which are usually more difficult to guarantee when
numerical methods are constructed without these restrictions.
On the other hand, not imposing SBP restrictions can possibly result in more
degrees of freedom which can be used to construct more flexible and possibly
a larger number of numerical methods.
From a practical point of view, the availability of numerical algorithms in
standard software packages and the efficiency of the implementations are also
very important. In this respect, established time integration methods have
definitely many advantages, since they are widespread and considerable efforts
went into the available implementations. Additionally, a practitioner
can chose to make a trade-off between guaranteed stability properties and the
efficiency of schemes that ``just work'' in practice, although only weaker
stability results might be available. For example, linearly implicit time
integration schemes such as Rosenbrock methods can be very efficient for certain
problems.

Having said all that, it is important to note that the process of discretizing
differential equations is filled with pitfalls. Potentially unstable schemes
may lead to results that seem correct but are in fact erroneous. A provably
stable scheme can be seen as a quality stamp.

\appendix

\section{Proof of Theorem~\ref{thm:SBP-proj-accuracy}}
\label{sec:proof-of-thm:SBP-proj-accuracy}

Here, we present the technical proof of Theorem~\ref{thm:SBP-proj-accuracy}.

\begin{proof}[Proof of Theorem~\ref{thm:SBP-proj-accuracy}]
  Consider the classical simplifying assumptions
  \begin{align}
  \label{eq:B-xi-block}
    B(\xi)\colon&&
    \vec{b}^T \vec{c}^{q-1} &= \frac{1}{q},
    && q \in \set{1, \dots, \xi},
    \\
  \label{eq:C-eta-block}
    C(\eta)\colon&&
    A \vec{c}^{q-1} &= \frac{1}{q} \vec{c}^q,
    && q \in \set{1, \dots, \eta},
    \\
  \label{eq:D-zeta-block}
    D(\zeta) \colon&&
    A^T B \vec{c}^{q-1} &= \frac{1}{q} B (\vec{1} - \vec{c}^q),
    && q \in \set{1, \dots, \zeta},
  \end{align}
  where $B = \mathrm{diag}(\vec{b})$ is a diagonal matrix.
  To prove Theorem~\ref{thm:SBP-proj-accuracy}, we will use the following
  result of Butcher \cite{butcher1964implicit} for Runge-Kutta methods.
  \begin{quote}
    If $B(\xi)$, $C(\eta)$, $D(\zeta)$, $\xi \le 1 + \eta + \zeta$, and
    $\xi \le 2 + 2\eta$, then the Runge-Kutta method has an order of
    accuracy at least $\xi$.
  \end{quote}

\begin{enumerate}[label=\alph*)]
  \item General mass matrices $M$
\begin{itemize}
  \item Proving $B(p)$

  The quadrature rule given by the weights $\vec{b} = \frac{1}{T} M \vec{1}$
  of the $p$th order accurate SBP operator is exact for polynomials of degree
  $p-1$ for a general norm matrix $M$ \cite[Theorem~1]{fernandez2014generalized}.
  Hence, $B(p)$ is satisfied.

  \item Proving $C(p)$

  As in the proof of Theorem~\ref{thm:Lobatto-IIIA}, $C(p)$ is satisfied
  by construction of $A = \frac{1}{T} \Dinv \F$, since all polynomials of degree
  $\le p-1$ are integrated exactly by $A$ with vanishing initial value at $t = 0$.

  \item Concluding

  Since only $B(p)$ is satisfied in general, the order of the Runge-Kutta method
  is limited to $p$ and we do not need further conditions on the simplifying
  assumption $D(\zeta)$. Hence, it suffices to consider the empty condition $D(0)$.
  Since $B(p)$, $C(p-1)$, and $D(0)$ are satisfied, the result of Butcher cited
  above implies that the Runge-Kutta method has at least an order of accuracy
  at least~$p$.
\end{itemize}

  \item Diagonal mass matrices $M$
\begin{itemize}
  \item Proving $B(2p)$

  For a diagonal norm matrix $M$, the quadrature rule given by the weights
  $\vec{b} = \frac{1}{T} M \vec{1}$ of the $p$th order accurate SBP operator
  is exact for polynomials of degree $2p-1$
  \cite[Theorem~2]{fernandez2014generalized}. Hence, $B(2p)$ is satisfied.

  \item Proving $C(p)$

  The proof of $C(p)$ is exactly the same as for general mass matrices $M$.

  \item Proving $D(p-1)$

  It suffices to consider a scaled SBP operator such that $T = 1$,
  $\vec{t}_L^T \vec{\tau} = 0$, $\vec{t}_R^T \vec{\tau} = 1$.
  Then, inserting the Runge-Kutta coefficients \eqref{eq:RK-characterization},
  $D(p-1)$ is satisfied if for all $q \in \set{1, \dots, p-1}$
  \begin{equation}
    \F^T \Dinv^T M \vec{\tau}^{q-1} = \frac{1}{q} M (\vec{1} - \vec{\tau}^q).
  \end{equation}
  This equation is satisfied if and only if for all $\vec{u} \in \R^s$
  \begin{equation}
    \vec{u}^T \F^T \Dinv^T M \vec{\tau}^{q-1}
    - \frac{1}{q} \vec{u}^T M (\vec{1} - \vec{\tau}^q)
    = 0.
  \end{equation}
  Every $\vec{u}$ can be written as $\vec{u} = D \vec{v} + \alpha \osc$,
  where $\vec{t}_L^T \vec{v} = 0$ and $\alpha \in \R$, since $D$ is nullspace
  consistent. Hence, it suffices to consider
  \begin{equation}
  \begin{aligned}
    &\quad
    \vec{v}^T D^T \F^T \Dinv^T M \vec{\tau}^{q-1}
    + \alpha \osc^T \F^T \Dinv^T M \vec{\tau}^{q-1}
    - \frac{1}{q} \vec{v}^T D^T M \vec{1}
    - \frac{1}{q} \alpha \osc^T M \vec{1}
    \\
    &\quad
    + \frac{1}{q} \vec{v}^T D^T M \vec{\tau}^q
    + \frac{1}{q} \alpha \osc^T M \vec{\tau}^q
    \\
    &=
    \vec{v}^T M \vec{\tau}^{q-1}
    - \frac{1}{q} \vec{v}^T D^T M \vec{1}
    + \frac{1}{q} \vec{v}^T D^T M \vec{\tau}^q.
  \end{aligned}
  \end{equation}
  Here, we used that $\Dinv \F D \vec{v} = \vec{v}$ by definition of
  $\vec{v} \in V_0 = \kernel \vec{t}_L^T$. The filter $\F$ removes grid
  oscillations, i.e.\ $\F \osc = \vec{0}$. Additionally, grid oscillations are
  orthogonal to the constant $\vec{1}$ for SBP operators that are at least
  first-order accurate and orthogonal to $\vec{\tau}^q$ for $q \le p-1$
  in general (since $D \vec{\tau}^p = p \vec{\tau}^{p-1}$).

  Using the SBP property \eqref{eq:SBP}, the expression above can be rewritten as
  \begin{equation}
  \begin{aligned}
    &\quad
    \vec{v}^T M \vec{\tau}^{q-1}
    - \frac{1}{q} \vec{v}^T D^T M \vec{1}
    + \frac{1}{q} \vec{v}^T D^T M \vec{\tau}^q
    \\
    &=
    \vec{v}^T M \vec{\tau}^{q-1}
    - \frac{1}{q} \vec{t}_R^T \vec{v}
    + \frac{1}{q} \vec{v}^T \vec{t}_R
    - \frac{1}{q} \vec{v}^T M D \vec{\tau}^q
    = 0,
  \end{aligned}
  \end{equation}
  where we used $\vec{t}_L^T \vec{v} = 0$, $\vec{t}_R^T \vec{\tau}^q = 1$,
  and the accuracy of $D$. This proves $D(p-1)$.

  \item Concluding

  Since $B(2p)$, $C(p)$, and $D(p-1)$ are satisfied, the result of Butcher
  cited above guarantees an order of accuracy of at least $2p$.
\end{itemize}
\end{enumerate}
\end{proof}

\begin{remark}
  The simplifying assumptions used in the proof of
  Theorem~\ref{thm:SBP-proj-accuracy} can be satisfied to higher order of
  accuracy. For example, the Lobatto~IIIA methods satisfy $C(p+1)$ instead
  of only $C(p)$, cf.\ the proof of Theorem~\ref{thm:Lobatto-IIIA}.
  In that case, the order of the associated quadrature still limits the order
  of the corresponding Runge-Kutta method to $2p$.

  Similarly, the method based on left Radau quadrature mentioned in
  Example~\ref{rem:other-classical-schemes} has Butcher coefficients
  \begin{equation}
    A =
    \begin{pmatrix}
      0 & 0 \\
      \nicefrac{1}{6} & \nicefrac{1}{2}
    \end{pmatrix},
    \quad
    b =
    \begin{pmatrix}
      \nicefrac{1}{4} \\
      \nicefrac{3}{4}
    \end{pmatrix},
    \quad
    c =
    \begin{pmatrix}
      0 \\
      \nicefrac{2}{3}
    \end{pmatrix}.
  \end{equation}
  Thus, the quadrature condition is satisfied to higher order of accuracy,
  i.e.\ $B(3)$ holds instead of only $B(2)$ for $p = 1$. Nevertheless,
  the Runge-Kutta method is only second-order accurate, i.e. it satisfies the
  order conditions
  \begin{equation}
    b^T \vec{1} = 1,
    \quad
    b^T A \vec{1} = \frac{1}{2},
  \end{equation}
  but violates one of the additional conditions for a third-order accurate method,
  i.e.
  \begin{equation}
    b^T (A \vec{1})^2 = \frac{1}{3},
    \quad
    b^T A^2 \vec{1} = \frac{1}{4} \neq \frac{1}{6},
  \end{equation}
  since neither $C(p+1)$ nor $D(p)$ is satisfied.
\end{remark}

\section{An Example using Gauss Quadrature}
\label{sec:example-gauss}

Here, we follow the derivation of the scheme presented in
Section~\ref{sec:new-schemes} using a more complicated SBP operator that does
not include any boundary node.
Consider the SBP operator of order $p = 2$ induced by classical Gauss-Legendre
quadrature on $[0, T]$, using the nodes
\begin{equation}
  \tau_1 = \frac{T}{10} (5 - \sqrt{15}),
  \quad
  \tau_2 = \frac{T}{2},
  \quad
  \tau_3 = \frac{T}{10} (5 + \sqrt{15}).
\end{equation}
The associated SBP operator exactly differentiating polynomials of degree $p = 2$
is given by
\begin{equation}
  D = \frac{\sqrt{15}}{3 T} \begin{pmatrix}
    -3 & 4 & -1 \\
    -1 & 0 & 1 \\
    1 & -4 & 3
  \end{pmatrix},
  \quad
  M = \frac{T}{18} \begin{pmatrix}
    5 \\
    & 8 \\
    && 5
  \end{pmatrix},
  \quad
  \vec{t}_L = \frac{1}{6} \begin{pmatrix}
    5 + \sqrt{15} \\
    -4 \\
    5 - \sqrt{15}
  \end{pmatrix},
  \quad
  \vec{t}_R = \frac{1}{6} \begin{pmatrix}
    5 - \sqrt{15} \\
    -4 \\
    5 + \sqrt{15}
  \end{pmatrix}.
\end{equation}
Therefore,
\begin{equation}
  D^* = M^{-1} D^T M = \frac{\sqrt{15}}{30 T} \begin{pmatrix}
    -30 & -16 & 10 \\
    25 & 0 & -25 \\
    -10 & 16 & 30
  \end{pmatrix}
\end{equation}
and $\kernel D^* = \spann\{ \osc \}$, where $\osc = (4, -5, 4)^T$
represents the highest resolvable grid oscillation; $\osc$ is orthogonal to
$\image D$, since $\osc^T M D = \vec{0}^T$.
The adjoint of $\osc$ is
\begin{equation}
  \osc^* = \osc^T M = \frac{10 T}{9} (1, -2, 1).
\end{equation}
Thus, the filter/projection operator \eqref{eq:F} is
\begin{equation}
  \F = \I - \frac{\osc \osc^*}{\norm{\osc}_M^2} =
  \frac{1}{18} \begin{pmatrix}
    14 & 8 & -4 \\
    5 & 8 & 5 \\
    -4 & 8 & 14
  \end{pmatrix}.
\end{equation}
The vector spaces $V_0$, $V_1$ associated to the given SBP operator are
\begin{equation}
  V_0 = \set{
    \begin{pmatrix}
      \alpha_1 \\
      \alpha_2 \\
      -(4 + \sqrt{15}) \alpha_1 + 2 (5 + \sqrt{15}) \alpha_2 / 5
    \end{pmatrix} | \alpha_1, \alpha_2 \in \R},
  \quad
  V_1 = \set{
    \begin{pmatrix}
      \beta_1 \\
      \beta_2 \\
      -\beta_1 + 2 \beta_2
    \end{pmatrix} | \beta_1, \beta_2 \in \R}.
\end{equation}
To verify this, observe that both vector spaces are two-dimensional,
$V_0$ is the nullspace of $\vec{t}_L^T$, and $V_1$ is the nullspace of $\osc^*$.

At the level of $\R^3$, the inverse $J$ of $D$ can be represented as
\begin{equation}
  J = \frac{3 T}{24 \sqrt{15}}
  \begin{pmatrix}
    2 & 4 (\sqrt{15} - 3) & -2 \\
    5 & 4 \sqrt{15} & -5 \\
    2 & 4 (\sqrt{15} + 3) & -2
  \end{pmatrix}.
\end{equation}
Indeed, for $\vec{u} = (
  \alpha_1,
  \alpha_2,
  -(4 + \sqrt{15}) \alpha_1 + 2 (5 + \sqrt{15}) \alpha_2 / 5
)^T$, $\alpha_1, \alpha_2 \in \R$,
\begin{equation}
  J D \vec{u}
  =
  \frac{1}{6}
  \begin{pmatrix}
     1 - \sqrt{15} &  4 & -5 + \sqrt{15} \\
    -5 - \sqrt{15} & 10 & -5 + \sqrt{15} \\
    -5 - \sqrt{15} &  4 &  1 + \sqrt{15}
  \end{pmatrix}
  \begin{pmatrix}
    \alpha_1 \\
    \alpha_2 \\
    -(4 + \sqrt{15}) \alpha_1 + 2 (5 + \sqrt{15}) \alpha_2 / 5
  \end{pmatrix}
  =
  \vec{u}
\end{equation}
and similarly for $\vec{u} = (
  \beta_1,
  \beta_2,
  -\beta_1 + 2 \beta_2
)^T$, $\beta_1, \beta_2 \in \R$,
\begin{equation}
  D J \vec{u}
  =
  \frac{1}{2} \begin{pmatrix}
    1 & 2 & -1 \\
    - & 2 & 0 \\
    -1 & 2 & 1
  \end{pmatrix}
  \begin{pmatrix}
    \beta_1 \\
    \beta_2 \\
    -\beta_1 + 2 \beta_2
  \end{pmatrix}
  =
  \vec{u}.
\end{equation}
Hence, $J D = \id_{V_0}$ and $D J = \id_{V_1}$, where $\id_{V_i}$ is the identity
on $V_i$.

Since the filter $\F$ removes the highest grid oscillations and maps a grid function
into the image of the derivative operator $D$, the inverse $J$ can be applied
after $\F$, resulting in
\begin{equation}
  J F =
  \frac{T}{72 \sqrt{15}} \begin{pmatrix}
    -12 + 10 \sqrt{15} & -48 + 16 \sqrt{15} & -48 + 10 \sqrt{15} \\
     45 + 10 \sqrt{15} &       16 \sqrt{15} & -45 + 10 \sqrt{15} \\
     48 + 10 \sqrt{15} &  48 + 16 \sqrt{15} &  12 + 10 \sqrt{15}
  \end{pmatrix}.
\end{equation}
Thus, the Butcher coefficients associated to the new SBP projection method
given by Theorem~\ref{thm:RK-characterization} are
\begin{equation}
\begin{aligned}
  A &= \frac{1}{T} J \F =
  \frac{1}{72 \sqrt{15}} \begin{pmatrix}
    -12 + 10 \sqrt{15} & -48 + 16 \sqrt{15} & -48 + 10 \sqrt{15} \\
     45 + 10 \sqrt{15} &       16 \sqrt{15} & -45 + 10 \sqrt{15} \\
     48 + 10 \sqrt{15} &  48 + 16 \sqrt{15} &  12 + 10 \sqrt{15}
  \end{pmatrix},
  \\
  b &= \frac{1}{T} M \vec{1} =
  \frac{1}{18}
  \begin{pmatrix}
    5 \\
    8 \\
    5
  \end{pmatrix},
  \qquad
  c = \frac{1}{T} (\tau_1, \tau_2, \tau_3)^T =
  \frac{1}{10}
  \begin{pmatrix}
    5 - \sqrt{15} \\
    5 \\
    5 + \sqrt{15}
  \end{pmatrix}.
\end{aligned}
\end{equation}
Note that the coefficients in $A$ are different from those of the classical
Runge-Kutta Gauss-Legendre collocation method (while $b$ and $c$ are the same
by construction), see also Remark~\ref{rem:not-Gauss} below. In particular, this
method is fourth-order accurate while the classical Gauss-Legendre collocation
method with the same number of stages is of order six.

\section{Lobatto~IIIB Schemes}
\label{sec:LobattoIIIB}

Another scheme similar to \eqref{eq:SBP-proj} can be constructed
by considering $-D$ as bijective operator acting on functions that
vanish at the right endpoint. The equivalent of $J$, the inverse of
$D$ mapping $\image D$ to the space of grid functions vanishing at the left
endpoint, in this context is written as $\mDinv$, which is the inverse
of $-D$ mapping $\image (-D)$ to the space of grid functions vanishing at the
right endpoint.
The corresponding matrix $A$ of the Runge-Kutta method becomes
\begin{equation}
\label{eq:SBP-proj-mD}
  A = A^{**} = M^{-1} \bigl( A^* \bigr)^T M,
  \quad
  A^* = \frac{1}{T} \mDinv \left( \I - \frac{\osc \osc^*}{\norm{\osc}_M^2} \right).
\end{equation}
Using $A$ of \eqref{eq:SBP-proj-mD} and $b, c$ as in
\eqref{eq:RK-characterization}, the scheme is defined as the Runge-Kutta
method \eqref{eq:RK-step} with these Butcher coefficients $A, b, c$.

\begin{theorem}
\label{thm:Lobatto-IIIB}
  If the SBP derivative operator $D$ is given by the
  nodal polynomial collocation scheme on Lobatto-Legendre nodes,
  the Runge-Kutta method \eqref{eq:RK-step} with $A$ as in
  \eqref{eq:SBP-proj-mD} and $b, c$ as in \eqref{eq:RK-characterization}
  is the classical Lobatto~IIIB method.
\end{theorem}
\begin{proof}
  The Lobatto~IIIB methods are given by the nodes $c$ and weights $b$
  of the Lobatto-Legendre quadrature, just as the SBP method.
  Hence, it remains to prove that the classical condition
  $D(s)$ is satisfied \cite[Section~344]{butcher2016numerical}, where
  \begin{equation}
  \label{eq:D-zeta}
    D(\zeta) \colon
    \qquad \sum_{i=1}^s b_i c_i^{q-1} a_{i,j} = \frac{1}{q} b_j (1 - c_j^q),
    \quad j \in \set{1, \dots, s},\; q \in \set{1, \dots, \zeta}.
  \end{equation}
  In matrix vector notation, this can be written as
  \begin{equation}
    A^T M \vec{c}^{q-1} = \frac{1}{q} M (\vec{1} - \vec{c}^q)
    \iff
    \underbrace{M^{-1} A^T M}_{=A^*} \vec{c}^{q-1} = \frac{1}{q} (\vec{1} - \vec{c}^q),
  \end{equation}
  where the exponentiation $\vec{c}^q$ is performed pointwise.

  In other words, all polynomials of degree $\leq p$ must be integrated exactly
  by $A^*$ with vanishing final value at $t = 1$. By construction of
  $A$, this is satisfied for all polynomials of
  degree $\leq p-1$, and the proof can be continued as the one of
  Theorem~\ref{thm:Lobatto-IIIA}.
\end{proof}

\begin{remark}
  Remarks~\ref{rem:L-B-stability} and \ref{rem:other-classical-schemes}
  hold analogously: Schemes based on \eqref{eq:SBP-proj-mD} are also in
  general not $L$ or $B$ stable and other classical schemes on Gauss, Radau,
  or Lobatto nodes are not included in this class.
\end{remark}

Similarly to Theorem~\ref{thm:SBP-proj-accuracy}, we present some results on
the order of accuracy of the new SBP methods.
Since this class of methods is made to satisfy the simplifying condition
$D(\zeta)$ instead of $C(\eta)$ and stronger results on $C(\eta)$ are necessary
to apply the results of Butcher \cite{butcher1964implicit}, we concentrate on
diagonal norms.

\begin{theorem}
\label{thm:SBP-proj-mD-accuracy}
  For nullspace consistent SBP operators that are $p$th order accurate with
  $p \ge 1$ and a diagonal norm matrix $M$, the Runge-Kutta method associated
  to the SBP time integration scheme \eqref{eq:SBP-proj-mD} has an order of
  accuracy of at least $2p$.
\end{theorem}
\begin{proof}
  This proof is very similar to the one of Theorem~\ref{thm:SBP-proj-accuracy}.
  By construction, the simplifying conditions $B(2p)$ and $D(p)$ are satisfied.
  Hence, the order of accuracy is at least $2p$ if $C(p-1)$ is satisfied.

  Again, suffices to consider a scaled SBP operator such that $T = 1$,
  $\vec{t}_L^T \vec{\tau} = 0$, $\vec{t}_R^T \vec{\tau} = 1$.
  Then, inserting the Runge-Kutta coefficients \eqref{eq:SBP-proj-mD},
  $C(p-1)$ is satisfied if for all $q \in \set{1, \dots, p-1}$
  \begin{equation}
    M^{-1} \F^T \mDinv M \vec{\tau}^{q-1} = \frac{1}{q} \vec{\tau}^q.
  \end{equation}
  This equation is satisfied if and only if for all $\vec{u} \in \R^s$
  \begin{equation}
    \vec{u}^T \F^T \mDinv^T M \vec{\tau}^{q-1}
    - \frac{1}{q} \vec{u}^T M \vec{\tau}^q
    = 0.
  \end{equation}
  Every $\vec{u}$ can be written as $\vec{u} = D \vec{v} + \alpha \osc$,
  where $\vec{t}_R^T \vec{v} = 0$ and $\alpha \in \R$, since $D$ is nullspace
  consistent. Hence, it suffices to consider
  \begin{equation}
  \begin{aligned}
    &\quad
    \vec{v}^T D^T \F^T \mDinv^T M \vec{\tau}^{q-1}
    \alpha \osc^T \F^T \mDinv^T M \vec{\tau}^{q-1}
    - \frac{1}{q} \vec{v}^T D^T M \vec{\tau}^q
    - \frac{1}{q} \alpha \osc^T M \vec{\tau}^q
    \\
    &=
    - \vec{v}^T M \vec{\tau}^{q-1}
    - \frac{1}{q} \vec{v}^T D^T M \vec{\tau}^q
    \\
    &=
    - \vec{v}^T M \vec{\tau}^{q-1}
    - \frac{1}{q} \vec{v}^T (\vec{t}_R \vec{t}_R^T - \vec{t}_L \vec{t}_L^T - M D) \vec{\tau}^q
    \\
    &=
    - \vec{v}^T M \vec{\tau}^{q-1}
    _ \frac{1}{q} \vec{v}^T - M D \vec{\tau}^q
    -
    0.
  \end{aligned}
  \end{equation}
  This proves $C(p-1)$.
\end{proof}

\begin{remark}
\label{rem:not-Gauss}
  Up to now, it is unclear whether the classical collocation Runge-Kutta
  schemes on Gauss nodes can be constructed as special members of a family
  of schemes that can be formulated for (more) general SBP operators. The
  problem seems to be that SBP schemes rely on differentiation, while the
  conditions $C(s)$ and $D(s)$ describing the Runge-Kutta methods rely on
  integration. Thus, special compatibility conditions as in the proof of
  Theorem~\ref{thm:Lobatto-IIIA} are necessary. To the authors' knowledge,
  no insights in this direction have been achieved for Gauss methods.
\end{remark}

\section{Butcher Coefficients of some Finite Difference SBP Methods}
\label{sec:Butcher-FD}

Here, we provide the Butcher coefficients given by
Theorem~\ref{thm:RK-characterization} for some of the finite difference SBP
methods used in the numerical experiments in Section~\ref{sec:numerical-experiments}.
\begin{itemize}
  \item
  Interior order 2, $N = 3$ nodes
  \begin{equation}
    A = \begin{pmatrix}
      0 & 0 & 0 \\
      \nicefrac{3}{8} & \nicefrac{1}{4} & \nicefrac{-1}{8} \\
      \nicefrac{1}{4} & \nicefrac{1}{2} & \nicefrac{1}{4} \\
    \end{pmatrix},
    \quad
    b = \begin{pmatrix}
      \nicefrac{1}{4} \\
      \nicefrac{1}{2} \\
      \nicefrac{1}{4} \\
    \end{pmatrix},
    \quad
    c = \begin{pmatrix}
      0 \\
      \nicefrac{1}{2} \\
      1 \\
    \end{pmatrix}
  \end{equation}

  \item
  Interior order 2, $N = 9$ nodes
  \begin{equation}
    A = \begin{pmatrix}
      0 & 0 & 0 & 0 & 0 & 0 & 0 & 0 & 0 \\
      \nicefrac{15}{128} & \nicefrac{1}{64} & \nicefrac{-1}{64} & \nicefrac{1}{64} & \nicefrac{-1}{64} & \nicefrac{1}{64} & \nicefrac{-1}{64} & \nicefrac{1}{64} & \nicefrac{-1}{128} \\
      \nicefrac{1}{64} & \nicefrac{7}{32} & \nicefrac{1}{32} & \nicefrac{-1}{32} & \nicefrac{1}{32} & \nicefrac{-1}{32} & \nicefrac{1}{32} & \nicefrac{-1}{32} & \nicefrac{1}{64} \\
      \nicefrac{13}{128} & \nicefrac{3}{64} & \nicefrac{13}{64} & \nicefrac{3}{64} & \nicefrac{-3}{64} & \nicefrac{3}{64} & \nicefrac{-3}{64} & \nicefrac{3}{64} & \nicefrac{-3}{128} \\
      \nicefrac{1}{32} & \nicefrac{3}{16} & \nicefrac{1}{16} & \nicefrac{3}{16} & \nicefrac{1}{16} & \nicefrac{-1}{16} & \nicefrac{1}{16} & \nicefrac{-1}{16} & \nicefrac{1}{32} \\
      \nicefrac{11}{128} & \nicefrac{5}{64} & \nicefrac{11}{64} & \nicefrac{5}{64} & \nicefrac{11}{64} & \nicefrac{5}{64} & \nicefrac{-5}{64} & \nicefrac{5}{64} & \nicefrac{-5}{128} \\
      \nicefrac{3}{64} & \nicefrac{5}{32} & \nicefrac{3}{32} & \nicefrac{5}{32} & \nicefrac{3}{32} & \nicefrac{5}{32} & \nicefrac{3}{32} & \nicefrac{-3}{32} & \nicefrac{3}{64} \\
      \nicefrac{9}{128} & \nicefrac{7}{64} & \nicefrac{9}{64} & \nicefrac{7}{64} & \nicefrac{9}{64} & \nicefrac{7}{64} & \nicefrac{9}{64} & \nicefrac{7}{64} & \nicefrac{-7}{128} \\
      \nicefrac{1}{16} & \nicefrac{1}{8} & \nicefrac{1}{8} & \nicefrac{1}{8} & \nicefrac{1}{8} & \nicefrac{1}{8} & \nicefrac{1}{8} & \nicefrac{1}{8} & \nicefrac{1}{16} \\
    \end{pmatrix},
    \quad
    b = \begin{pmatrix}
      \nicefrac{1}{16} \\
      \nicefrac{1}{8} \\
      \nicefrac{1}{8} \\
      \nicefrac{1}{8} \\
      \nicefrac{1}{8} \\
      \nicefrac{1}{8} \\
      \nicefrac{1}{8} \\
      \nicefrac{1}{8} \\
      \nicefrac{1}{16} \\
    \end{pmatrix},
    \quad
    c = \begin{pmatrix}
      0 \\
      \nicefrac{1}{8} \\
      \nicefrac{1}{4} \\
      \nicefrac{3}{8} \\
      \nicefrac{1}{2} \\
      \nicefrac{5}{8} \\
      \nicefrac{3}{4} \\
      \nicefrac{7}{8} \\
      1 \\
    \end{pmatrix}
  \end{equation}

  \item
  Interior order 4, $N = 9$ nodes
  \begin{equation}
    A = \begin{pmatrix}
      0 & 0 & 0 & 0 & 0 & 0 & 0 & 0 & 0 \\
      \nicefrac{13}{180} & \nicefrac{18}{385} & \nicefrac{1}{2044} & \nicefrac{2}{211} & \nicefrac{-3}{371} & \nicefrac{1}{124} & \nicefrac{-3}{317} & \nicefrac{2}{215} & \nicefrac{-1}{267} \\
      \nicefrac{5}{434} & \nicefrac{60}{271} & \nicefrac{3}{103} & \nicefrac{-7}{283} & \nicefrac{3}{118} & \nicefrac{-7}{283} & \nicefrac{3}{103} & \nicefrac{-20}{699} & \nicefrac{5}{434} \\
      \nicefrac{17}{265} & \nicefrac{37}{361} & \nicefrac{37}{228} & \nicefrac{13}{230} & \nicefrac{-4}{157} & \nicefrac{7}{244} & \nicefrac{-7}{211} & \nicefrac{3}{92} & \nicefrac{-4}{305} \\
      \nicefrac{11}{408} & \nicefrac{11}{56} & \nicefrac{47}{689} & \nicefrac{99}{614} & \nicefrac{1}{16} & \nicefrac{-11}{327} & \nicefrac{13}{297} & \nicefrac{-8}{187} & \nicefrac{5}{289} \\
      \nicefrac{7}{122} & \nicefrac{42}{347} & \nicefrac{109}{751} & \nicefrac{37}{374} & \nicefrac{31}{206} & \nicefrac{29}{408} & \nicefrac{-8}{159} & \nicefrac{20}{391} & \nicefrac{-7}{352} \\
      \nicefrac{15}{458} & \nicefrac{39}{214} & \nicefrac{29}{350} & \nicefrac{39}{256} & \nicefrac{24}{241} & \nicefrac{39}{256} & \nicefrac{29}{350} & \nicefrac{-21}{310} & \nicefrac{15}{458} \\
      \nicefrac{23}{479} & \nicefrac{55}{381} & \nicefrac{17}{140} & \nicefrac{41}{343} & \nicefrac{35}{263} & \nicefrac{43}{364} & \nicefrac{32}{287} & \nicefrac{31}{290} & \nicefrac{-9}{322} \\
      \nicefrac{12}{271} & \nicefrac{57}{371} & \nicefrac{43}{384} & \nicefrac{43}{337} & \nicefrac{1}{8} & \nicefrac{43}{337} & \nicefrac{43}{384} & \nicefrac{57}{371} & \nicefrac{12}{271} \\
    \end{pmatrix},
    \;
    b = \begin{pmatrix}
      \nicefrac{17}{384} \\
      \nicefrac{59}{384} \\
      \nicefrac{43}{384} \\
      \nicefrac{49}{384} \\
      \nicefrac{1}{8} \\
      \nicefrac{49}{384} \\
      \nicefrac{43}{384} \\
      \nicefrac{59}{384} \\
      \nicefrac{17}{384} \\
    \end{pmatrix},
    \;
    c = \begin{pmatrix}
      0 \\
      \nicefrac{1}{8} \\
      \nicefrac{1}{4} \\
      \nicefrac{3}{8} \\
      \nicefrac{1}{2} \\
      \nicefrac{5}{8} \\
      \nicefrac{3}{4} \\
      \nicefrac{7}{8} \\
      1 \\
    \end{pmatrix}
  \end{equation}
\end{itemize}

\section*{Acknowledgments}

Research reported in this publication was supported by the
King Abdullah University of Science and Technology (KAUST).
Jan Nordström was supported by Vetenskapsrådet, Sweden grant 2018-05084 VR
and by the Swedish e-Science Research Center (SeRC) through project ABL in SESSI.

\printbibliography

\end{document}